\documentclass{amsart}
\usepackage[latin1]{inputenc}
\usepackage{amsmath}
\usepackage{amsfonts}
\usepackage{amssymb}
\usepackage[english]{babel}
\usepackage{enumerate}
\usepackage{subfig}
\usepackage{caption}
\usepackage{graphicx}
\usepackage{color}
\usepackage{hyperref}
\def\RR{{\mathbb R}}
\def\ZZ{{\mathbb Z}}
\def\C{{\mathbb C}}
\def\CC{{\mathbb C}}
\def\ph{{\mathit p^{(h)}}}
\def\Rhc{{\mathcal R}^{(h)}}
\def\Rhcfr{{\mathcal R}^{(h),fr}}

\def\NN{{\mathbb N}}\def\RR{{\mathbb R}}
\def\conv{\text{\normalfont conv}}

\newcommand{\omh}{\omega^{(h)}}
\newcommand{\Rh}{R^{(h)}}
\newcommand{\Rhi}{R^{(h)}_\infty}
\def\Jo{J_\omega}

\def\black{\color{black}}

\newcommand{\etalchar}[1]{$^{#1}$}

\newcommand{\xh}[1]{\mathbf x^{(h)}_{#1}}
\newcommand{\Aha}[1]{A^{(h)}_\omega}

\newtheorem{Theorem}{Theorem}

\newtheorem{Lemma}[Theorem]{Lemma}

\newtheorem{Remark}[Theorem]{Remark}

\newtheorem{Example}[Theorem]{Example}

\newtheorem{Proposition}[Theorem]{Proposition}
\newtheorem{Notation}[Theorem]{Notation}

\author{Anna Chiara Lai
  \and Paola Loreti}
\title[Robot's hand and expansions in non-integer bases]{Robot's hand and expansions in non-integer bases}
\address{Dipartimento di Scienze di Base e Applicate per l'Ingegneria\\
Sezione di Matematica\\
Sapienza Universit\`a di Roma, Italy}
\begin{document}

\maketitle

\begin{abstract}\black We study a robot hand model in the
framework of the theory of expansions in non-integer bases. We
investigate the reachable workspace and we study some configurations enjoying form closure properties.
\end{abstract}\black

\section{Introduction}
Aim of this paper is to give a model of a robot's hand based on the theory of expansions
in non-integer bases. Self-similarity of configurations and an
arbitrarily large number of fingers (including the opposable thumb) and
phalanxes are the main features.  Binary controls rule the dynamics of
the hand, in particular the extension and the rotation of each phalanx.

 Our robot hand is composed by an arbitrary number of fingers, including the opposable thumb. 
Each finger moves on a plane. Every plane is assumed to be parallel to the others, excepting the thumb and the index finger, that belong to the same plane. 

A discrete dynamical system models the position of the extremal
junction of every finger. A \emph{configuration} is a sequence of states of
the system corresponding to a particular choice for the controls,
while the union of all the possible states of the system is named
\emph{reachable workspace for the finger}. The closure of the reachable workspace is named
\emph{asymptotic reachable workspace}. 
Our model includes two binary control parameters on every phalanx of
every finger of the robot hand. The first control parameter rules the length of
the phalanx, that can be either $0$ or a fixed value, while the
other control rules the angle between the current phalanx and the
previous one. Such an angle can be either $\pi$, namely the phalanx
is consecutive to the previous, or a fixed angle $\pi-\omega\in(0,\pi)$.

The structure of the finger ensures the set of possible configurations to be self-similar. In particular the sub-configurations can be looked at as scaled miniatures
with constant ratio $\rho$, named \emph{scaling factor}, of the whole structure. This is the key idea underlying our model and our main tool of investigation. 

We establish a connection between our model and  the theory of iterated function systems and the theory of expansions in non-integer bases. 
This yields several results describing the reachable workspace, some conditions on the parameters in order to avoid self-intersecting configurations 
and a description of a class of configurations satisfying a form closure condition.

\subsection{Previous work and motivations}

 The fingers of our robot hand are planar manipulators with rigid links and with a (arbitrarily) large number of degrees of freedom, 
that is they belong to the class of so-called macroscopically-serial hyper-redundant manipulators (the term was first introduced in \cite{term}).

 Hyper-redundant architecture was intensively studied back to the late 60's, when the first prototype of 
hyper-redundant robot arm was built \cite{snake}. 
The interest of researchers in devices with redundant controls was motivated, among others, by
the ability to avoid obstacles and the ability to perform new forms of robot locomotion and grasping (see for instance
\cite{Bai86}, \cite{Bur88} and \cite{CB95}).

A large number of papers were devoted in the literature to both continuously and discretely controlled hyper-redundant manipulators. 
Our approach, based on discrete actuators, is motivated by their precision with low
cost compared to actuators with continuous range-of-motion. Moreover the resulting discrete space of configurations reduces the cost of 
 position sensors and feedbacks.

In \cite{IC96} the inverse kinematics of discrete hyper-redundant manipulators is investigated. Throughout the analysis of the reachable workspace (and in particular of the density of its points)
an algorithm solving the inverse kinematics problem in linear time with respect the number of actuators is introduced. In general the number of points of 
the reachable workspace increases exponentially, the computational cost on the optimization of the density distribution of the workspace is investigated in \cite{LSD02}

Note that the concept of a binary tree describing all the possible configurations underlies above mentioned approaches, in our method the self-similar structure of such a tree gives access to
well-established results on fractal geometry and iterated function systems theory.  Robotic devices with a similar fractal structure are described in \cite{NASA}.  

Other approaches to the investigation of the reachable workspace include those based on  harmonic analysis \cite{Chr96}, and  Fast Fourier Transform
\cite{WC04}. Finally we refer to \cite{Chr00} for a description of the geometry of the reachable workspace.

In our model every link (phalanx) is controlled by a couple of binary controls. The control of the rotation at every joint is a common feature of all above mentioned manipulators. The study of a control ruling
the extension of every link has twofold applications. In one hand it can be physically implemented by
 means of telescopic links, that are particularly efficient in constrained workspaces (see \cite{tele}).
 On the other hand, our model can be considered a discrete approximation of continuous snake-like manipulators - see for instance the approach in \cite{snake1} to the discretization 
of a continuous curve and its applications to snake-like robots.

Form closure is a property of grasping mechanisms originally investigated in \cite{formclosure1} and it concerns the ability of the manipulator of totally or partially 
constrain the motion of a manipulated object. We refer to \cite{Bic95} for an overview on the closure properties of manipulators and their applications. 
In our analysis we shall focus on the case of planar manipulators constraining a circle without considering external forces. Some numerical examples are given for three-dimensional manipulators constraining a cylinder. 
This investigation has twofold motivations: in one hand, in our model the length of links decreases exponentially and we want to understand  by an explicit geometric argument 
how much this assumption affects the capability of the manipulator of interacting with other objects. On the other hand this simplified setting enlightens how self-similarity can
be used in order to extend local geometric properties to a wider class of configurations. We remark that a more complete analysis 
would involve the  investigation of the stronger condition of stable grasping, by also considering external forces and by handling sliding and rotational symmetries. This is however beyond the purposes of present investigation which is mostly concerned with the relations between geometrical properties
and self-similarity of configurations.

Some of our theoretical tools come from the theory of non-integer bases. For an overview on this topic we refer to \cite{Ren57}, \cite{Par60}, \cite{EK98} and to the book \cite{DK02a}. 
In particular, expansions in non-
integer bases were introduced in \cite{Ren57}. For the geometrical
aspects of the expansions in complex base namely the arguments that
are more related to our problem, we refer to \cite{Knu60}, \cite{Pen65},\cite{GG79},
\cite{Gil81},\cite{Gil87},\cite{akisurvey},\cite{akitop3} and to \cite{IKR92}.

\subsection{Organization of present paper}
The paper is organized as follows. In Section 2 we introduce the model and in Section 3 we remark its relation with the theory of non-integer number systems. 
Self-similarity of the configurations and some reachability results are showed in Section 4. 
In Section 5 we discuss a necessary and sufficient condition to avoid self-intersecting configurations in a particular case.
Form closure properties  are finally investigated in Section \ref{sectiongrasp}.

\section{The model}
 In our model the robot hand is composed by
$H$ fingers, every finger has an arbitrary number of phalanxes. We
assume junctions and phalanxes of each finger to be thin, so to be
respectively approximated with their middle axes and barycentres and
we also assume the junctions of every finger to be coplanar.
Inspired by the human hand, we set the fingers of our robot as
follows: the first two fingers are coplanar and they have in common
their first junction (they are our robotic version of the thumb and
the index finger of the human hand) while the remaining $H-1$ fingers
belong to parallel planes. By choosing an appropriate coordinate system $oxyz$ we may assume that the the first two fingers belong to the plane
$\mathit p^{(1)}: z=0$ while, for $h\geq 2$, $h$-th finger belongs to the plane $\mathit p^{(h)}: z=z_0^{(h)}$ for some $z_0^{(2)},\dots,z_0^{(H)}\in\RR$.

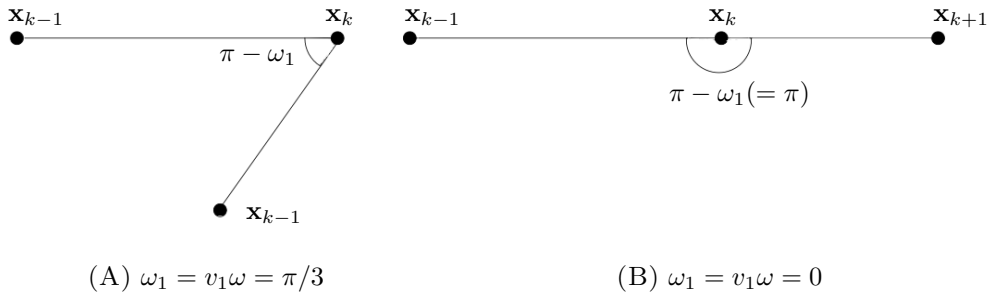
\begin{figure}
\begin{center}
\begin{picture}(300,150)
\put(-30,0){\includegraphics[scale=0.6]{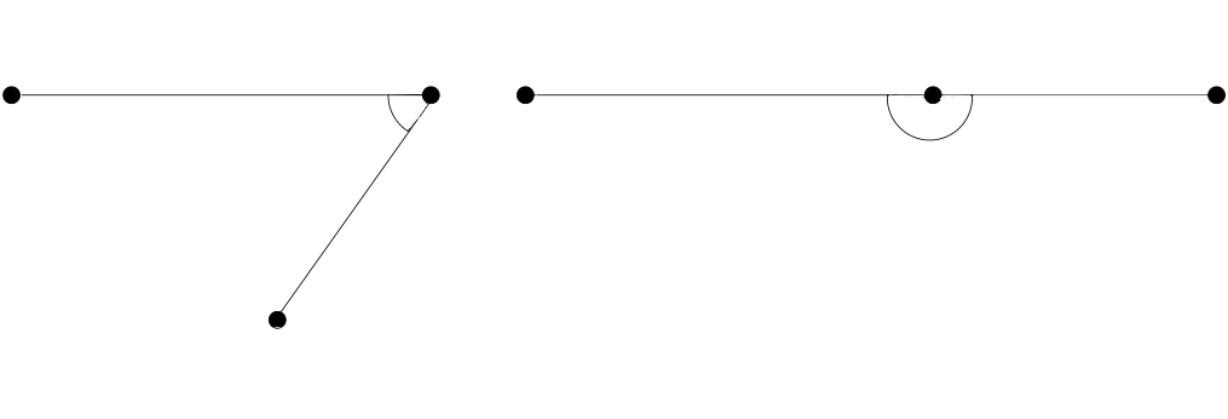}}
\put(0,0){\textsc{(A)} $\omega_1=v_1 \omega=\pi/3$}
\put(200,0){\textsc{(B)} $\omega_1=v_1 \omega=0$}
 \put(-30,100){$\mathbf x_{k-1}$}
 \put(120,100){$\mathbf x_{k-1}$}
\put(90,100){$\mathbf x_{k}$}
\put(50,85){$\pi-\omega_1$}
\put(220,70){$\pi-\omega_1(=\pi)$}
\put(235,100){$\mathbf x_{k}$}
\put(320,100){$\mathbf x_{k+1}$}
 \put(60,25){$\mathbf x_{k-1}$}
 \end{picture}
\end{center}
 \caption{\label{figexample}In both cases $u_{k}=u_{k+1}=1$, $\omega=\pi/3$.}
\end{figure}

\begin{figure}
\begin{center}
\begin{picture}
(0,130)
\hskip-0.5cm\subfloat[$v_{k+1}=0$;]{
\includegraphics[scale=0.6]{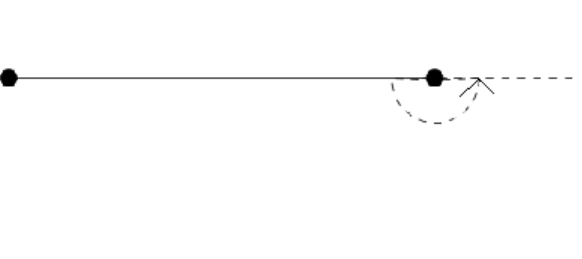}}
\put(-170,60){$\xh{k-1}$}
\put(-60,60){$\xh{k}\equiv \xh{k+1}$}
\put(-45,30){$\pi$}
 \end{picture}
\hskip6cm
\subfloat[$v_{k+1}=1$.]{
\includegraphics[scale=0.6]{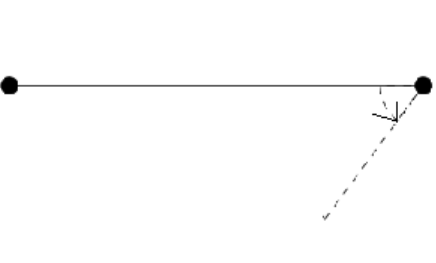}}
\put(-130,60){$\xh{k-1}$}
\put(-20,60){$\xh{k}\equiv \xh{k+1}$}
\put(-45,40){$\pi-\omega$}
 \end{center}

 \caption{\label{aff1}In both cases $u_{k+1}=0$.}
\end{figure}
\begin{figure}[ht]
\begin{center}
\hskip-2.5cm\subfloat[$v_{k+1}=0$;]{
\includegraphics[scale=0.6]{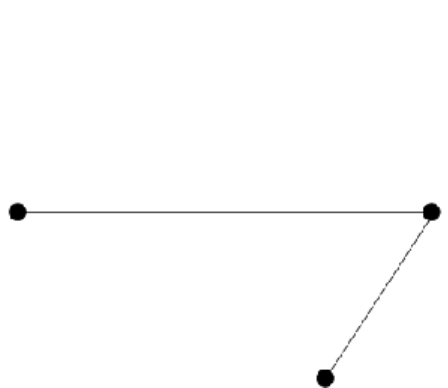}}
\put(-130,57){$\xh{k-1}$}
\put(-25,57){$\xh{k}\equiv \xh{k+1}$}
\put(-55,10){$\xh{k+2}$}
 \hskip4.5cm
\begin{picture}(0,150)
\hskip-3.2cm\subfloat[$v_{k+1}=1$.]{
\includegraphics[scale=0.6]{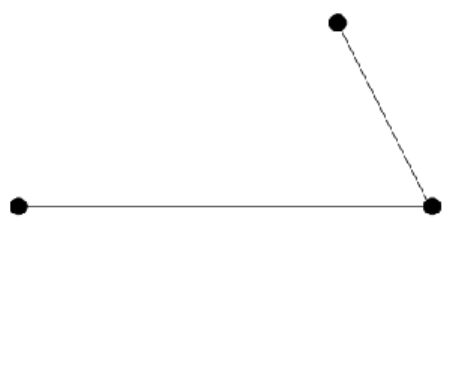}}
\put(-132,42){$\xh{k-1}$}
\put(-24,42){$\xh{k}\equiv \xh{k+1}$}
\put(-63,105){$\xh{k+2}$}
 \end{picture}
\end{center}
 \caption{\label{aff2}In both cases $u_{k+1}=0$, $u_{k+2}=1$ and $v_{k+2}=1$.}
\end{figure}

We now describe in more detail the model of a robot finger. A
configuration of a finger is the sequence $(\mathbf x_k)_{k=0}^K\subset
\RR^3$ of its junctions. The configurations of every finger are
ruled by two phalanx-at-phalanx motions: extension and rotation. In
particular, the length of $k$-th phalanx of the finger is either $0$
or $\dfrac{1}{\rho^k}$, where $\rho>1$ is a fixed ratio: this choice
is ruled by the a binary control we denote by using the symbol
$u_k$, so that the length $l_k$ of the $k$-th phalanx is
$$l_k:=\parallel \mathbf x_k-\mathbf x_{k-1}\parallel=\frac{u_k}{\rho^k}.$$

As all the phalanxes of a finger belong to the same plane, say $\mathit p$, in order
to describe the angle between two consecutive phalanxes, say the
$k-1$-th and the $k$-th phalanx, we just need to consider a
one-dimensional parameter, $\omega_k$. Each phalanx can lay
on the same line as the former or it can form with it a fixed planar angle $\omega\in(0,\pi)$, whose
vertex is the $k-1$-th junction. In other words, two consecutive phalanxes form either the angle $\pi$ or $\pi-\omega$. By introducing the binary control
$v_k$ we have that the angle between the $k-1$-th and $k$-th phalanx
is $\pi-\omega_k$, where
$$\omega_k=v_k\omega.$$

See Figure \ref{figexample} for the general case and Figure \ref{aff1} and Figure \ref{aff2} for the case $u_k=0$.

To describe the kinematic of the finger we adopt the Denavit-Hartenberg (DH) convention.
To this end, first of all recall that our base coordinate frame $oxyz$ is such that $oxy$ is parallel to $p$ (hence to every plane $\mathit p^{(h)}$) 
and we consider the finger coordinate frame $o_0x_0y_0z_0$ associated to the $4\times 4$ homogeneous transform 
$$A_0=\begin{pmatrix}
       \cos \omega_0 &-\sin \omega_0& 0&x_0\\
       \sin \omega_0 &\cos \omega_0& 0&y_0\\
       0 &0& 1&z_0\\
       0 &0& 0&1
      \end{pmatrix}$$
for some $\omega_0\in[0,2\pi)$. In particular if $\mathbf x$ and $\mathbf x_0$ are respectively coordinates of a point with respect to $oxyz$ and $o_0x_0y_0z_0$ then
$$\begin{pmatrix}
   \mathbf x\\
1
  \end{pmatrix}=A_0 \begin{pmatrix}
\mathbf x^{0}\\
1
\end{pmatrix}.$$
\begin{Remark}
 When only one finger is considered one may assume the base coordinate frame to coincide with the finger coordinate frame: this reduces $A_0$ to the identity and it could be omitted it in the model.
 The need of a coordinate frame for the finger rises when more than one finger, especially in the case of co-planar, opposable fingers, is considered. 
\end{Remark}

Now, the (DH) method consists in attaching to every phalanx, say the $k$-th phalanx, a coordinate frame $o_kx_ky_kz_k$, 
so that $\mathbf x_k$ coincides with $o_k$ and $\mathbf x_k-\mathbf x_{k-1}$ is parallel to $o_kx_k$ (see Figure \ref{figref}).
 Note that the coordinates of $\mathbf x_{k+1}$ with respect to
$o_kx_ky_kz_k$ are $(\frac{u_{k+1}}{\rho^{k+1}}\cos \omega_{k+1},\frac{u_{k+1}}{\rho^{k+1}}\sin \omega_{k+1},0)$.

 Since we are considering a planar manipulator, for every $k>1$ the geometric relation between the coordinate systems the $k-1$-th and the $k$-th phalanx is expressed by the matrix
$$A_k:=\begin{pmatrix}
       \cos \omega_k &-\sin \omega_k& 0&\frac{u_k}{\rho^k}\cos \omega_k\\
       \sin \omega_k &\cos \omega_k& 0&-\frac{u_k}{\rho^k}\sin \omega_k\\
       0 &0& 1&0\\
       0 &0& 0&1
      \end{pmatrix}$$
where the rotation matrix
$$\begin{pmatrix}
       \cos \omega_k &-\sin \omega_k& 0\\
       \sin \omega_k &\cos \omega_k& 0\\
       0 &0& 1
      \end{pmatrix}$$
represents the rotation of the coordinate frame $o_kx_ky_kz_k$ with respect to $o_{k-1}x_{k-1}y_{k-1}z_{k-1}$ and the vector $(\frac{u_k}{\rho^k}\cos \omega_k,-\frac{u_k}{\rho^k}\sin \omega_k,0)$ represents
the position of $o_k$ with respect to $o_{k-1}x_{k-1}y_{k-1}z_{k-1}$. 

Set
$$T_k:=\prod_{j=0}^k A_j.$$
By definition $T_k$ is the composition the transforms $A_0,\dots,A_k$ and, consequently, it represents the relation between the base coordinate frame $oxyz$ and $o_kx_ky_kz_k$.
In particular
$$T_k=\begin{pmatrix}
       R_k& P_k\\
       0 &1
      \end{pmatrix}.$$
where $R_k$ is a $3\times 3$ rotation matrix and the entries of the vector $P_k$  are the coordinates of $o_k(=\mathbf x_k)$ in the reference system $oxyz$. 
Expliciting $T_k$ one has
$$R_k=\begin{pmatrix}
       \cos\left(\sum_{j=0}^k \omega_j\right)&-\sin\left(\sum_{j=0}^k \omega_j\right)&0\\
       \sin\left(\sum_{j=0}^k \omega_j\right)&\cos\left(\sum_{j=0}^k \omega_j\right)&0\\
0&0&1
      \end{pmatrix}
$$
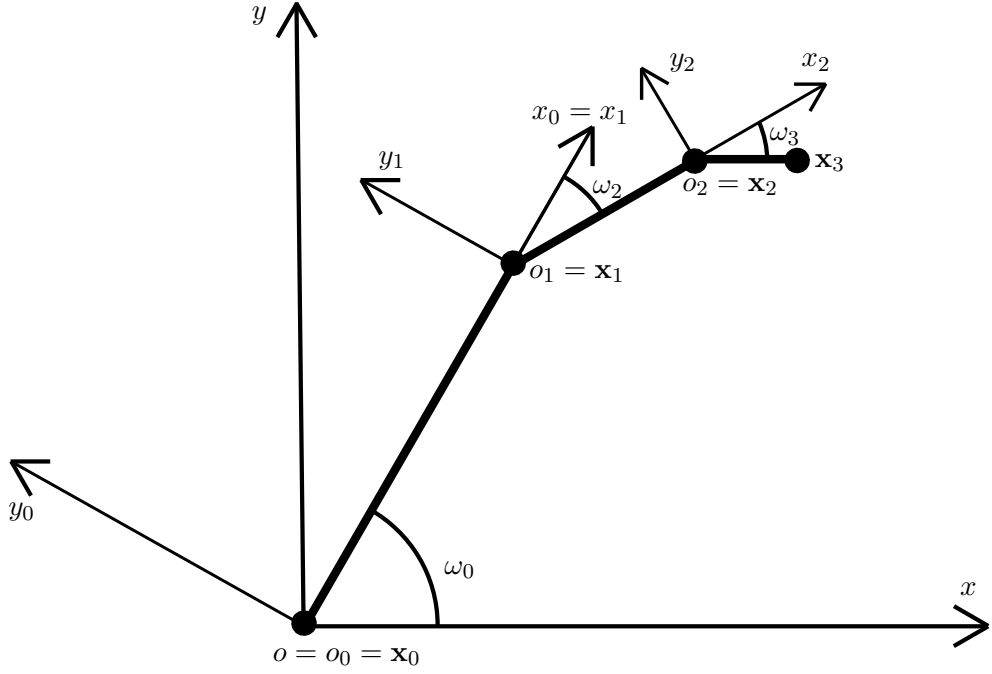
\begin{figure}
\vskip-0.5cm \begin{picture}(300,250)
\put(0,0){\includegraphics[scale=2]{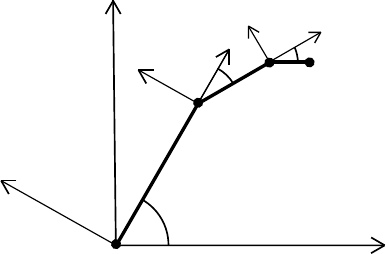}}
\put(0,50){\large$y_0$}
\put(360,20){\large $x$}
\put(92,238){\large $y$}
\put(100,-5){\large$o=o_0=\mathbf x_0$}
\put(165,28){\large$\omega_0$}
\put(250,220){\large$y_2$}
\put(305,182){\large$\mathbf x_3$}
\put(288,192){\large$\omega_3$}
\put(198,200){\large$x_0=x_1$}
\put(197,140){\large$o_1=\mathbf x_1$}
\put(221,172){\large$\omega_2$}
\put(300,220){\large$ x_2$}
\put(140,183){\large$ y_1$}
\put(255,173){\large$ o_2=\mathbf x_2$}
 \end{picture}
\caption{\label{figref} A finger with rotation control vector $(0,1,1)$ (in particular $\omega_1=0$ and $\omega_2=\omega_3=\pi/6$) and extension control vector $(1,1,1)$}
\end{figure}

and
$$P_k=P_0+\sum_{j=1}^k R_j \begin{pmatrix} 
\frac{u_j}{\rho^j}\\
0\\
0
\end{pmatrix}=
\begin{pmatrix}
 x_0+\sum_{j=1}^k \frac{u_j}{\rho^j}\cos\left(\sum_{n=0}^j \omega_n\right)\\
 y_0-\sum_{j=1}^k \frac{u_j}{\rho^j}\sin\left(\sum_{n=0}^j \omega_n\right)\\
z_0
\end{pmatrix}
$$
Then for every $k\geq0$
$$\mathbf x_k=P_k=\begin{pmatrix}
 x_0+\sum_{j=0}^k \frac{u_j}{\rho^j}\cos\left(\sum_{n=0}^j \omega_j\right)\\
 y_0-\sum_{j=0}^k \frac{u_j}{\rho^j}\sin\left(\sum_{n=0}^j \omega_j\right)\\
z_0\end{pmatrix}.$$

\subsection{Reachable workspace}
From now on we shall assume $(x_0,y_0,z_0)=(0,0,z_0)$, so that given a couple of control vectors $\mathbf u=(u_j)_{j=1}^k$ and $\mathbf v=(v_j)_{j=1}^k$ one has
$$\mathbf x_k=\mathbf x_k(\mathbf u,\mathbf v)=\begin{pmatrix}
 \sum_{j=1}^k \frac{u_j}{\rho^j}\cos\left(\sum_{n=0}^j \omega_j\right)\\
 \sum_{j=1}^k \frac{u_j}{\rho^j}\sin\left(\sum_{n=0}^j \omega_j\right)\\
z_0\end{pmatrix}.$$

Now, we index the fingers of our hand by $h\in \{1,\dots,H\}$, so that the $k$-th junction of the $h$-th finger reads $\mathbf x_k^{(h)}$, the scaling ratio and the maximal rotation angle 
respectively read $\rho^{(h)}$ and $\omega^{(h)}$ (so that $\omega_k^{(h)}=\omega^{(h)}v_k$), the orientation of the $h$-finger with respect to the
base reference frame is $\omega_0^{(h)}$, the $z$ coordinate of every junction is $z_0^{(h)}$. 

We also define $\Omega^{(h)}_j(\mathbf v):=\sum_{n=1}^j \omega^{(h)}_n=\sum_{n=1}^j \omega^{(h)}v_n$ so that one has for every $k\geq 1$
$$\mathbf x_k^{(h)}=\xh{k}(\mathbf u,\mathbf v)=\begin{pmatrix}
 \displaystyle{\sum_{j=0}^k \dfrac{u_j}{(\rho^{(h)})^j}\cos\left(\omega_0^{(h)}+\Omega^{(h)}_j(\mathbf v)\right)}\\
 \displaystyle{-\sum_{j=0}^k \dfrac{u_j}{(\rho^{(h)})^j}\sin\left(\omega_0^{(h)}+\Omega^{(h)}_j(\mathbf v)\right)}\\
z_0^{(h)}\end{pmatrix}.$$

A point $\mathbf x=(x,y,z)\in\RR^3$  belongs to \emph{reachable workspace of the $h$-th finger} if there exists a couple of control vectors $\mathbf u=(u_j)_{j=1}^k$ and $\mathbf v=(v_j)_{j=1}^k$ such that
\begin{equation}
\mathbf x=\mathbf x^{(h)}(\mathbf u,\mathbf v).
\end{equation}
The point $\mathbf x$ belongs to the \emph{asymptotically reachable workspace of the $h$-th finger} if there exists a couple of (infinite) control vectors $(\mathbf u,\mathbf v)=((u_j)_{j\geq1},(v_j)_{j\geq1})\in\{0,1\}^\NN\times\{0,1\}^\NN$ 
satisfying
\begin{equation}
\mathbf x=\lim_{k\to \infty} \xh{k}(\mathbf u,\mathbf v)=\begin{pmatrix}
 \displaystyle{\sum_{j=1}^\infty \dfrac{u_j}{(\rho^{(h)})^j}\cos\left(\omega_0^{(h)}+\Omega^{(h)}_j(\mathbf v)\right)}\\
 \displaystyle{\sum_{j=1}^\infty \dfrac{u_j}{(\rho^{(h)})^j}\sin\left(\omega_0^{(h)}+\Omega_j^{(h)}(\mathbf v)\right)}\\
z_0^{(h)}\end{pmatrix}.
\end{equation}
We use the symbols $R^{(h)}$ and $R^{(h)}_\infty$ to respectively denote set of reachable and asymptotically reachable workspace with respect the $h$-th finger of the hand. 

We also define 
$$\Rh_k:=\left\{\mathbf x_k(\mathbf u,\mathbf v)\mid (\mathbf u,\mathbf v)\in \{0,1\}^k\times \{0,1\}^k\right\};$$
\begin{Remark}
The following relations hold
\begin{equation}
 \Rh=\bigcup_{k=0}^\infty \Rh_k;
\end{equation}
\begin{equation}
 \Rhi=\overline{\Rh};
\end{equation}
in particular for every point in the asymptotically reachable workspace there exists an arbitrarily close element of the reachable workspace.
\end{Remark}
Finally we call \emph{reachable workspace} (resp. \emph{asymptotically reachable workspace}) 
the set 
$$R:=\bigcup_{h=0}^H \Rh\quad (resp. 
R_\infty:=\bigcup_{h=0}^H \Rhi)$$

\begin{Remark}
 As we assumed all the phalanxes of a fixed finger to be coplanar, we have
$$R\subset \bigcup_{h=1}^H \ph,$$
where $\ph$ is the plane of the $h$-finger. We remark that there are only $H$ distinct planes because we assumed the first two fingers, the thumb and the forefinger, to belong to the same plane $\mathit p^{(1)}$.
\end{Remark}

\section{Robot's hand and expansions in complex bases}
In this section we discuss reachability  in the framework of expansions in complex bases. Given a complex number $\lambda$ greater than $1$ in modulus and a possibly infinite set 
$A\subset \CC$ we say that $z\in \CC$ is \emph{representable} in \emph{base} $\lambda$ and with \emph{alphabet} $A$ if there exists a sequence $(z_j)_{j\geq1}$ of digits of $A$ such that
$$z=\sum_{j=1}^\infty \frac{z_j}{\lambda^j}.$$
A digit sequence $(z_j)_{j\geq1}$ satisfying the above equality is called \emph{expansion} of $z$ in base $\lambda$ and with alphabet $A$. 
It is well-known that coplanar rotations, like the ones performed by each finger of our hand, 
can be read as products on the complex plane. Therefore to perform infinite rotations and scalings 
(like in the case of asymptotic reachability problem) equals to consider complex-based power series and, consequently, expansions in non-integer bases. 
In what follows we formalize this concept.

Fix $k$ and a couple of binary control vectors $(\mathbf u,\mathbf v)$ and note that setting 
$$c_j=c^{(h)}_j(\mathbf v):=e^{-i(\omega_0^{(h)}+\Omega_j^{(h)}(\mathbf v))}$$
 one has
$$\mathbf x^{(h)}_k=\begin{pmatrix}
 \displaystyle{\sum_{j=0}^k \frac{u_j}{(\rho^{(h)})^j}\Re(c_j)}\\
 \displaystyle{\sum_{j=0}^k \frac{u_j}{(\rho^{(h)})^j}\Im(c_j)}\\
z_0^{(h)}\end{pmatrix}$$

Consequently, the reachable workspace $\Rh$ satisfies
$$\Rh=\left\{\displaystyle{\left(\Re(z),\Im(z),z_0^{(h)}\right)\mid z=\sum_{j=1}^k \frac{u_j}{(\rho^{(h)})^j}c_j(\mathbf v)};(\mathbf u,\mathbf v)\in \{0,1\}^k\times \{0,1\}^k;~k\in\NN\right\}$$
and the study of the reachable workspace becomes equivalent to the study of the sets of complex numbers
$$\Rhc:=\left\{\sum_{j=1}^k \frac{u_j}{(\rho^{(h)})^j}e^{-i\Omega_j(\mathbf v)};(\mathbf u,\mathbf v)\in\{0,1\}^k\times \{0,1\}^k;~k\in\NN\right\},$$
Indeed 
$$\Rh=\left\{\left(\Re(e^{-i\omega_0}z),\Im(e^{-i\omega_0}z),z_0^{(h)}\right)\mid z\in \Rhc\right\}.$$
We also define
$$ \Rhc_k:=\left\{\sum_{j=1}^k \frac{u_j}{(\rho^{(h)})^j}e^{-i\Omega_j(\mathbf v)};~(\mathbf u,\mathbf v)\in\{0,1\}^k\times \{0,1\}^k\right\}$$
and
$$\Rhc_\infty:=\left\{\sum_{j=1}^\infty \frac{u_j}{(\rho^{(h)})^j}e^{-i\Omega(\mathbf v)};~(\mathbf u,\mathbf v)\in\{0,1\}^\NN\times \{0,1\}^\NN\right\}.$$
Now, set $\lambda=\lambda^{(h)}:=\rho^{(h)} e^{i\omh}$ and consider the digit set $A_j:=\{0,e^{i(j\omega-\Omega_j)}\}=\{0,e^{iN\omega},N\in\{0,1,\dots,j\}\}$, we have
$$\Rhc_\infty=\left\{\sum_{j=1}^\infty \frac{z_j}{\lambda^j}\mid z_j \in A_j;~ j\in\NN\right\}.$$
and, setting, $A:=\bigcup_j A_j$
$$\Rhc_\infty\subseteq\left\{\sum_{j=1}^\infty \frac{z_j}{\lambda^j}\mid z_j \in A;~ j\in\NN\right\}.$$
 consequently \emph{the asymptotic reachable workspace of a finger is a subset of the points whose two first coordinates are the real and imaginary part of representable numbers in base $\lambda$ and with alphabet $A$}.

\begin{Remark}\label{rmkA}
If $\omh=\frac{p}{q}2\pi$ for some $p,q\in \ZZ$ then $A$ is a finite set.

If we restrict to full-rotation configurations, namely when the rotation controls are constantly equal to $1$, we have 
$$\Omega_j=\sum_{n=1}^jv_n\omh=j\omega$$ and, consequently, $A=A_j=\{0,1\}$ for every $j$.

In the case of full-extension configurations $A$ does not contain $0$.
\end{Remark}

\section{Some features of the asymptotic reachable workspace}
\subsection{Self-similarity}\label{sself}
It is well-known that all the sets of representable numbers with positional numbers system are self-similar, in particular they are the unique fixed point of appropriate linear iterated function systems (see for instance \cite{Gil87}). 
For a general introduction on fractal geometry and, in particular on fractals generated by iterated function systems, we refer to \cite{Fal90}.

We recall that an iterated function system (IFS) is a finite set of contractive functions $f_j:\CC\to\CC$. Every IFS admits a unique closed bounded set $R$, the \emph{attractor}, 
such that $R=\mathcal F(R)$, with respect to the \emph{Hutchinson operator} 
$$\mathcal F: S\mapsto \bigcup_{j=1}^J f_j(S)$$
 In particular there exists a set $R\subseteq\CC$ s.t. for every $S\subseteq \CC$
$$\lim_{k\to\infty} \mathcal F^k(S)=R,$$
in the Hausdorff metric, i.e. $R$ is the attractor of $\mathcal F$. 

  The asymptotic reachable workspace has this property, as well. In particular we have the following result, whose proof can be found in \cite{LL11}.
\begin{Proposition}\label{pss}
 For every  $\rho>1$ and $\omega\in(0,\pi)$, the asymptotic reachable workspace $\Rhc_{\infty}(\rho,\omega)$ is the (unique) fixed point of the IFS
\begin{equation}\label{mathf}
\mathcal F_{\rho,\omega}=\{f_h:\C\to\C\mid h=1,\dots,4\} 
\end{equation}

where
\begin{center}
\begin{equation}\label{pattrac}
\begin{tabular}{ll}
$ f_1: x\mapsto \displaystyle\frac{1}{\rho}x \qquad$ & $f_2: x\mapsto \displaystyle\frac{e^{-i(\pi-\omega)}}{\rho}x$\\
 $f_3: x\mapsto \displaystyle\frac{1}{\rho}(x+1) \qquad$ & $f_4: x\mapsto \displaystyle\frac{e^{-i(\pi-\omega)}}{\rho}(x+1)$.
\end{tabular}
\end{equation}
\end{center}
\end{Proposition}

\subsection{Reachability}

As we noticed in Remark \ref{rmkA}, the set of points that can be asymptotically reached with full-rotation 
configurations corresponds to the set of representable numbers with a suitable base and with alphabet
$A=\{0,1\}$, in particular 
$$\Rhcfr:=\left\{\sum_{j=1}^\infty \frac{z_j}{\lambda^j}\mid z_j\in\{0,1\}\right\}\subset \Rhc.$$
 This gives access to several results on complex-based representability that can be adapted to our case. Set $\CC_R:=\{z\in\CC\mid |z|\leq R\}$ with $R>0$ and $\omega \in(0,2\pi)$. 
In \cite{KL07} is shown that if $\rho$ is sufficiently close to $1$ then every complex number $z\in\CC_R$ has at least one expansion in base $\lambda=\rho e^{i\omega}$ and with
alphabet $\{0,1\}$, i.e., every point in $\CC_R$ can be reached by a full-rotation configuration. 
Moreover if $\omega=\frac{1}{q}2\pi$, with $q\in\NN$ and $q\geq 3$, and $\rho\leq 2^{1/q}$, then $\Rhcfr$ is a polygon with $2q$ edges if $q$ is odd and $q$ edges otherwise \cite{Lai11}.

We also remark that in \cite{akitop} and \cite{akitop2}  can be found a study of the topological properties of the so-called \emph{fundamental domains} of two-dimensional expansions: our set 
$\Rhcfr$ is a particular case of such domains. We refer to \cite{akitop3} as a survey on this argument. 

\section{How to avoid self-intersecting configurations: a particular case}
In general, the dynamics of the fingers does not prevent self-intersecting configurations (see Figure \ref{fself}) and, clearly, this needs to be avoided in order to keep the physical sense of the model.
In this section we show sufficient conditions to avoid self-intersecting configurations in a particular case, namely when the angle between phalanxes is $\pi/3$, i.e., $\omega=2\pi/3$. Our starting point is next result, whose proof can be found in \cite{LL11}.
\begin{Lemma}\label{lsic}
 For every $\rho>1$ 
\begin{equation}
 \conv(\Rhc_\infty(\rho,2\pi/3))=\conv(\{\mathbf v_1(\rho),\mathbf v_2(\rho),\mathbf v_3(\rho),\mathbf v_4(\rho)\})
\end{equation}
where
\begin{center}
 \begin{tabular}{ll}
$\mathbf v_1(\rho)=\dfrac{1}{\rho-1},$&$\qquad\mathbf v_2(\rho)=\dfrac{e^{-i\frac{2\pi}{3}}}{\rho-1},$\\
$\mathbf v_3(\rho)=\dfrac{e^{-\frac{2\pi}{3}i}}{\rho}+\dfrac{e^{-i\frac{4\pi}{3}}}{\rho(\rho-1)},$&$\qquad\mathbf v_4(\rho)=\dfrac{e^{-i\frac{4\pi}{3}}}{\rho(\rho-1)}.$
 \end{tabular}
\end{center}
\end{Lemma}

\begin{figure}[ht]\begin{center}
               \includegraphics[scale=0.3]{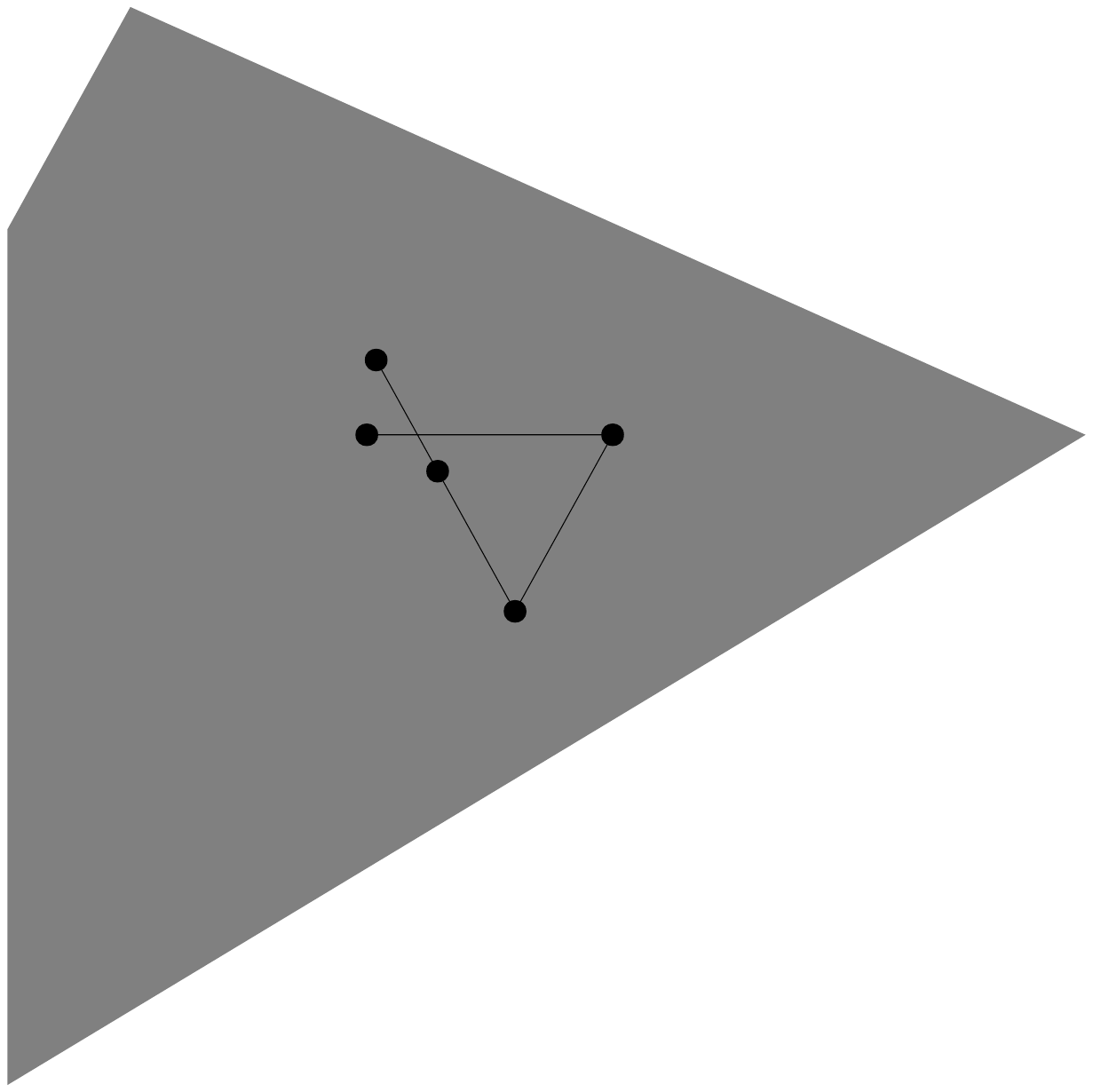}
              \end{center}
\caption{A self-intersecting configuration with $\rho=1.5$ and $\omega=\pi/3$\label{fself}}
\end{figure}

\begin{Theorem}\label{thintersect}
   If the angle between the phalanxes is $\pi/3$ and if the ratio is $\rho\geq2$, then all the configurations are not  self-intersecting.
\end{Theorem}

\begin{proof}
 A configuration is a finite subsequence of junctions of a finger on the complex plane and, consequently, every configuration is a scaled and rotated copy of a 
configuration starting from the initial junction
$x_0$ and with the first extended phalanx parallel to the real axis (see Proposition \ref{pss}). 
Therefore we may consider without loss of generality only configurations of the form $(x_j)_{0\leq j<\infty}$ with $x_0=0$ and with $x_1=\frac{1}{\rho}$
and it suffices to prove that any subsequent phalanx, namely the segment joining two consecutive junctions, does not intersect the first one. 
Let $J$ be the smallest integer such that $x_J\not=x_1$. We have
$$x_J=x_1+\frac{1}{\rho^J}e^{-i\sum_{n=1}^J v_n \pi/3}=x_1+\frac{1}{\rho^J}e^{-iN\pi/3}$$
for some $N\in\{0,1,2\}$. The asymptotic reachable workspace from $x_J$ is hence the following
$$\Rhc_\infty(x_J)=x_J+\frac{1}{\rho^J}e^{-iN\pi/3}\Rhc_\infty$$
and, in view of Lemma \ref{lsic}, 
$$\conv(\Rhc_\infty(x_J))=\conv(\{x_J+\frac{1}{\rho^J}e^{-iN\pi/3}\mathbf v_j(\rho)\mid j=1,\dots,4\}.$$
By a direct computation, we have that the intersection between $\conv(\Rhc_\infty(x_J))$ and the first phalanx $\{x_0+tx_1\mid t\in[0,1]\}$ is empty if $\rho>2$. 
In particular, if $\rho>2$ and if $N=0$ then the real part of every reachable point is greater than $\dfrac{1}{\rho}$, namely greater than one of the endpoints of the phalanx, 
if $N=1$ or $N=2$ then the imaginary part of any reachable point is respectively strictly smaller or greater than $0$.  When $\rho=2$ only infinite full-extension configurations
intersect the first phalanx. Since configurations are finite sequences, this proves the ``if part'' of the theorem.
 If $\rho<2$  then a direct computation shows that the configuration 
generated by the control vector $(u_j)_{j=1}^J$ and $(v_j)_{j=1}^J$ with $u_j=1$ for every $j=1,\dots,J$ and $v_10$, $v_2=v_3=1$ and $v_j=0$ for every $j=3,\dots,J$ is self-intersecting for every sufficiently 
large $J$ (see Figure \ref{fself}).
\end{proof}
\section{Form closure properties}\label{sectiongrasp}
 Let $\{\mathbf c_j\}$ with $j\in J\subset\NN$ be a set of contact points (namely a set of tangency
points between the finger and the surface of an object $\mathcal O$), let $\mathbf n_j$ be the normal vector to the boundary of $\mathcal O$ at $\mathbf c_j$ and let $\mathbf l$ a vector describing the linear velocity of $\mathcal O$ (and consequently
the linear velocity of any $\mathbf c_j$). Given the contact constraints
\begin{equation}\label{fc}
\mathbf  n_j^T\cdot \mathbf l \geq 0 \quad j\in J 
\end{equation}
we consider the following partial form closure condition: 
\begin{equation}\label{c}\tag{C}
 \text{\em there exists at most one unit vector $\mathbf l\in\RR^2$ satisfying the contact constraints (\ref{fc}).}
\end{equation}
In other words, if a configuration and an object satisfy (C) then the object cannot move in any the direction different from $\mathbf l$.  
We also consider a stronger version of (C)
\begin{equation}\label{sc}\tag{SC}
 \text{\em for every $\mathbf l\in\RR^2$ at least one of the contact constraints (\ref{fc}) is violated.}
\end{equation}

In our model we assume to actively control the modulus $\alpha_j\in[0,1]$ of some internal (squeezing) frictionless contact forces $\mathbf f_j$ so that 
\begin{equation}\label{force}
\mathbf f_j=-\alpha_j \mathbf n_j\quad \text{for every}~j\in J.
\end{equation}
setting $\mathbf w_j:=(\mathbf f_j,\mathbf m_j)$, where $\mathbf m_j$ is the momentum of $\mathbf f_j$, we also look for configurations satisfying the following equilibrium condition
\begin{equation}\label{E}\tag{E}
\sum_{j\in J} \mathbf w_j=0. 
\end{equation}
%

Now, let $\omega\in(0,\pi)$ be the angle of rotation of a finger and let $\Jo$ be the smallest integer such that
\begin{equation}\label{g1}
 \omega(\Jo-1)<\pi\leq \omega \Jo \quad \mod 2\pi
\end{equation}

\begin{Example}
 If $\omega=2\pi/3$ then $\Jo=2$. 
\end{Example}

Consider the configuration whose motion controls satisfy for every $0\leq j\leq J$, for some $J\geq\Jo$,
\begin{equation}\label{controls}
u_j\dot=\begin{cases}
       1 & \text{ if } j=1,\Jo,\Jo+1;\\
       0 & \text{otherwise;}
      \end{cases}
\qquad
v_j\dot=\begin{cases} 1 &\text{ if } j\leq \Jo+1;\\
     0 &\text{ otherwise.}
    \end{cases}
\end{equation}

\begin{Theorem}\label{thgrasp}
Let  $\omega\in(0,\pi)$ and consider the configuration corresponding to the controls defined in (\ref{controls}). 
Then there exists a circle $\mathcal O$ sharing with the finger three contact points $\mathbf c_1,\mathbf c_{\Jo}$ and $\mathbf c_{\Jo+1}$ and satisfying (C) if and only if 
\begin{equation}\label{claim}
 \rho< 2+\tan(\omega(\Jo-1)/2)\cot(\omega/2).
\end{equation}
If moreover $\omega\Jo\not=\pi$ then (\ref{claim}) also implies (SC).
\end{Theorem}

\begin{figure}[ht]
 \begin{center}
\subfloat[$\omega=2\pi/5$]{
\includegraphics[scale=0.5]{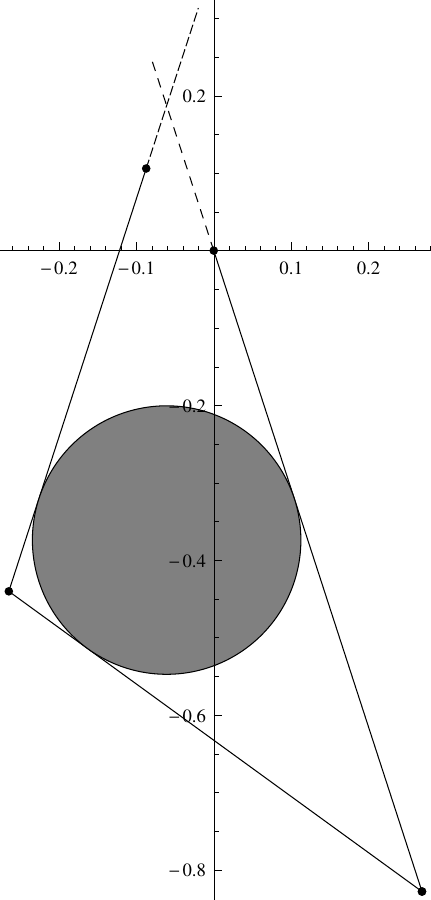}
}\hskip1.5cm
\subfloat[$\omega=2\pi/8$]{
\includegraphics[scale=0.5]{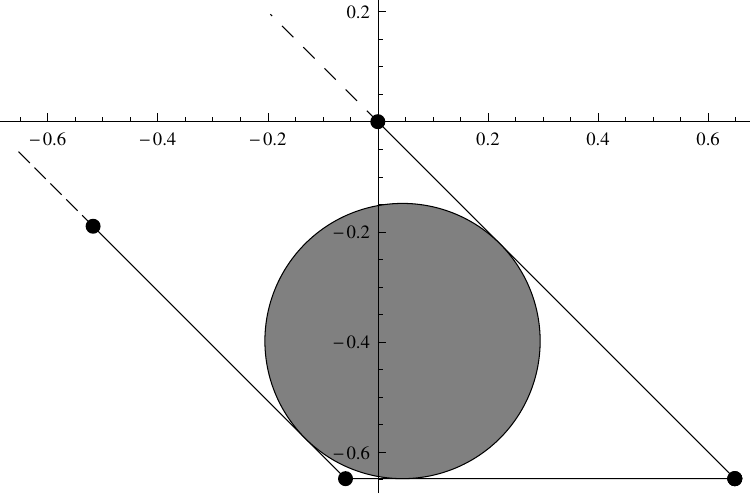}
}
 \end{center}
\caption{\label{fw54} Configurations associated to the controls defined in (\ref{controls}) and related inscribed circles $\mathcal O$.}
\end{figure}
\begin{proof}
  Let $J\geq J_\omega$ and note that (\ref{controls}) implies that the extended phalanxes of the resulting configuration are the first one and every phalanx between the $\Jo$-th and $J$-th ones.
 Moreover all phalanxes between the $\Jo+1$-th and the last belong to the same line, because their rotation controls are constantly $0$. 
Hence we may construct a circle $\mathcal O$ tangent to the prolongations of these phalanxes. In particular, 
in Figure \ref{fw54} are represented two possible scenarios: if $\Jo\omega\not=\pi$ (see Figure \ref{fw54}.A) then by construction the prolongations of the phalanxes form a triangle and we set
 $\mathcal O$ as the inscribed circle of this triangle. Note that the tangency points $\mathbf c_1,\mathbf c_{\Jo}$ and $\mathbf c_{\Jo+1}$ are indeed contact points (namely they belong to the phalanxes and
not to their prolongations) then by construction (SC) is satisfied. If otherwise $\Jo\omega=\pi$, then every extended phalanx but the second one is parallel to the first phalanx, 
see Figure \ref{fw54}.B. In this case we set $\mathcal O$ as the (unique) circle tangent to three distinct extended phalanxes: in particular, 
it is the circle inscribed in the rhombus whose edges are as long as the $\Jo$-th phalanx and whose internal angles are $\omega$ and $\pi-\omega$. Note that in this case the only allowed
 direction $\mathbf l$ (namely the only unit vector satisfying (\ref{fc})) is the one parallel to the first phalanx and with positive scalar product with the $J$-th phalanx.  
Also in this case call the resulting tangent points $\mathbf c_1$, $\mathbf c_{\Jo}$ and $\mathbf c_{\Jo+1}$.

It is left to show that (\ref{claim}) holds if and only if  $\mathbf c_1$, $\mathbf c_{\Jo}$ and $\mathbf c_{\Jo+1}$
 respectively belong to the first, to the $\Jo$-th and to any subsequent phalanx.  
Remark that the $\Jo$-th phalanx shares both its endpoints with other phalanxes, hence the prolongations we are
considering only refer to the first and to the last phalanxes: by construction $\mathbf c_{\Jo}$ is always tangent to the $\Jo$-th phalanx. In particular we have that the distance between $\mathbf c_{\Jo}$ and the $\Jo$-th 
junction is lower than the length of the $\Jo$-th phalanx:
\begin{equation}\label{cj}
\mid x_{1}-\mathbf c_{\Jo}\mid \leq \frac{1}{\rho^{\Jo}}.
\end{equation}
 Now, $\mathbf c_1$ belongs to the first phalanx if and only if
\begin{equation}\label{c01}
 \mid x_1-\mathbf c_1\mid \leq \frac{1}{\rho}
\end{equation}
where $x_1$ is the position of the first junction and $\frac{1}{\rho}$ is the length of the first phalanx. Similarly $\mathbf c_{\Jo+1}$ belongs to a phalanx if and only if 
\begin{equation}\label{c02}
 \mid x_{\Jo}-\mathbf c_{\Jo+1}\mid < \frac{1}{\rho^{\Jo}(\rho-1)}
\end{equation}
indeed the right-hand side of the above inequality is the upper bound of the length of a finite sequence of adjacent phalanxes. 
A classical result in plane geometry states that if we consider two consecutive
edges of a polygon admitting an inscribed circle, then the distances between the related tangent points and the common vertex are equal (see Figure \ref{plane}).

\begin{figure}[ht]
\begin{center}
 \begin{picture}(100,230)
\put(0,0){\includegraphics[scale=0.5]{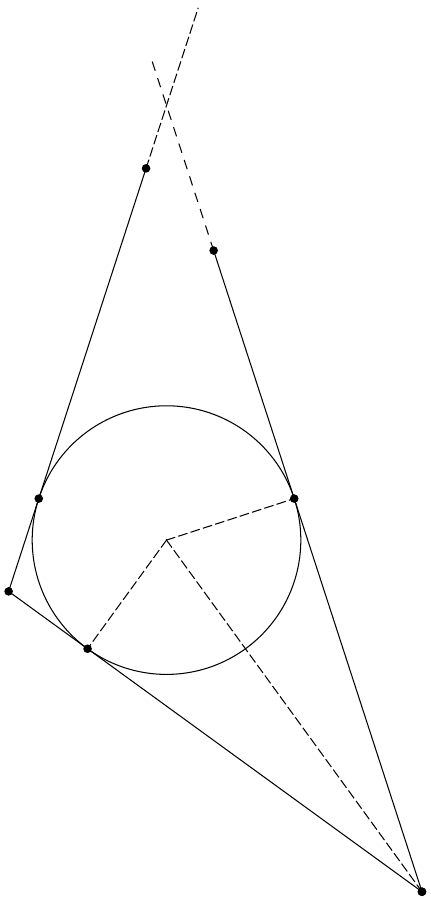}}
\put(105,0){$x_1$}
\put(-15,70){$x_{\Jo}$}
\put(15,50){$\mathbf c_{\Jo}$}
\put(30,85){$C$}
\put(72,95){$\mathbf c_1$}
\put(55,155){$x_0$}
\put(5,175){$x_{\Jo+1}$}
\put(-20,100){$\mathbf c_{\Jo+1}$}
 \end{picture}
\end{center}
\caption{\label{plane}The triangle  with vertices $C,~\mathbf c_{1}$ and $x_{1}$ and the 
triangle with vertices $C,~\mathbf c_{\Jo}$ and $x_{1}$ are equal, because they are right triangles with a common edge and with two equal edges (the ray of the inscribed circle). Hence
$\mid x_{\Jo}-\mathbf c_{\Jo}\mid = \mid x_{\Jo}-\mathbf c_{\Jo+1}\mid$. This also implies that the edge with endpoints $x_{1}$ and $C$ is the bisector of the angle in $x_{1}$.}
\end{figure}
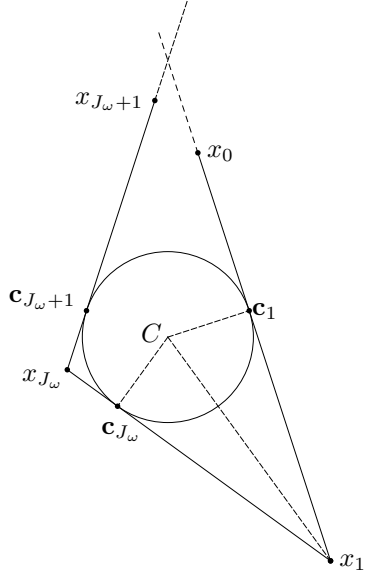

 In our case this implies, together with
(\ref{cj}),
\begin{equation}\label{c001}
 \mid x_1-\mathbf c_1\mid = \mid x_1-\mathbf c_{\Jo}\mid < \frac{1}{\rho}.
\end{equation}
Similarly we may rewrite (\ref{c02}) as follows
\begin{equation}\label{c002}
 \mid x_{\Jo}-\mathbf c_{\Jo}\mid < \frac{1}{\rho^{\Jo}(\rho-1)}.
\end{equation}
Now, the angle between the first two active phalanxes is $\pi-(\Jo-1)\omega$; therefore if we call $r$ the ray of the inscribed circle we have
\begin{equation}\label{comp1}
 \mid x_{\Jo}-\mathbf c_{\Jo}\mid = \frac{1}{\rho^{\Jo}}-\mid x_{1}-\mathbf c_{\Jo}\mid = \frac{1}{\rho^{\Jo}}-r \mid \tan((\pi-(\Jo-1)\omega)/2)\mid 
\end{equation}
(see Figure \ref{plane}).
Since the angle in the junction $x_{\Jo}$ is $\pi-\omega$, we also have
\begin{equation}\label{comp2}
 \mid x_{\Jo}-\mathbf c_{\Jo}\mid= r \mid\tan((\pi-\omega)/2)\mid.
\end{equation}
By a comparison between (\ref{comp1}) and (\ref{comp2}) we deduce
\begin{equation*}
 \mid x_{\Jo}-\mathbf c_{\Jo}\mid= \frac{1}{\rho^{\Jo}}\frac{\tan(\omega/2)}{\tan(\omega/2)+\tan((\Jo-1)\omega/2)}.
\end{equation*}
and, finally, the equivalence between (\ref{c02}) and (\ref{claim}).
\end{proof}

\begin{Example}\label{exgrasp}
If $\omega=2\pi/3$ then (\ref{claim}) holds if $\rho<3$.\\
If $\omega=2\pi/5$ then (\ref{claim}) holds if  $\rho<2G+3$, where $G=(1+\sqrt{5})/2$ is the Golden Ratio.
\end{Example}

\begin{Theorem}\label{kineto}
 Let $\mathcal O$ be like in Theorem \ref{thgrasp}. Then there exist $\alpha_1,\alpha_{\Jo}$ and $\alpha_{\Jo+1}\in[0,1]$ such that the equilibrium condition (E) is satisfied.
\end{Theorem}
\begin{Remark}
 Theorem \ref{kineto} straightforward follows by kineto-static duality, ensuring that (SC) is equivalent to the existence of 
a set of strictly compressive, normal, frinctionless contact forces preserving the equilibrium, see for instance \cite{duality}, Part D, Chapter 28. 
However we give an explicit proof of above result in order to keep the present paper as self-contained as possible.
\end{Remark}

\begin{proof}
 In the proof of Theorem \ref{thgrasp} we constructed $\mathcal O$ so that the contact with the finger involves the first, the $\Jo$-th phalanx and
a particular subsequent phalanx. Recall that we defined $\mathbf n_1$, $\mathbf n_{\Jo}$ and $\mathbf n_{\Jo+1}$ as the external normal versors to the boundary of the circle in the contact points. 
Since the contact phalanxes are tangent to the circle, $\mathbf n_1$, $\mathbf n_{\Jo}$ and $\mathbf n_{\Jo+1}$ are also normal to the first, the $\Jo$-th and to the $\Jo+1$-th phalanx.
 This allows us to explicitly determine $\mathbf n_1$, $\mathbf n_{\Jo}$ and $\mathbf n_{\Jo+1}$, in particular we may assume without loss of generality the finger to belong to the $xy$-plane in $\RR^3$ 
and get
\begin{equation}
\begin{split}
&\mathbf n_1=(\Re(z_1),\Im(z_1),0) \quad\text{ where } z_1=-e^{i(-\omega+\pi/2)};~ \\
&\mathbf n_{\Jo} =(\Re(z_{\Jo}),\Im(z_{\Jo}),0) \quad\text{ where } z_{\Jo}=-e^{i(-\Jo\omega+\pi/2)};~ \\
&\mathbf n_{\Jo+1}=(\Re(z_{\Jo+1}),\Im(z_{\Jo+1}),0) \quad\text{ where } z_{\Jo+1}=e^{i(-(\Jo+1)\omega+\pi/2)}.  
\end{split}
\end{equation}
so that 
\begin{equation}
\mathbf f_j=-\alpha_j \mathbf n_j; \quad j\in\{1,\Jo,\Jo+1\}.
\end{equation}
with $\alpha_1,\alpha_{\Jo},\alpha_{\Jo+1}\in[0,1]$. We call $C$ the center of the circle and we assume it to coincide with its barycenter, so that the resulting moments are
$$\mathbf m_j=(\mathbf c_j-C)\times \mathbf f_j;~\quad j\in \{1,\Jo,\Jo+1\}$$
We split (E) in the following conditions
\begin{equation}\label{feq}
 \mathbf f_1+\mathbf f_{\Jo}+\mathbf f_{\Jo}=0
\end{equation}
and
\begin{equation}\label{meq}
 \mathbf m_1+\mathbf m_{\Jo}+\mathbf m_{\Jo}=0
\end{equation}
Now, (\ref{feq}) can be set in the complex plane, and in particular we obtain the complex equation
\begin{equation*}
 \alpha_1 e^{i(-\omega+\pi/2)} +\alpha_{\Jo}e^{i(-\Jo\omega+\pi/2)}+\alpha_{\Jo+1}e^{i(-(\Jo+1)\omega+\pi/2)}=0
\end{equation*}
whose parametric solutions are
\begin{equation}
 \begin{cases}
  &\displaystyle\alpha_1(t)=t;\\\\
  &\displaystyle\alpha_{\Jo}(t)=-t \frac{\sin(\omega \Jo)}{\sin\omega};\\\\
  &\displaystyle\alpha_{\Jo+1}(t)= t \frac{\sin((\Jo-1)\omega)}{\sin\omega};\\
 \end{cases}
\end{equation}
with $t\in\RR$.
Setting 
$$t=\begin{cases}
     1 & \quad \text{ if } \Jo\omega=\pi;\\
     \displaystyle\min\left\{1,-\frac{\sin \omega}{\sin(\Jo\omega)},\frac{\sin \omega}{\sin((\Jo-1)\omega)}\right\} & \quad \text{ otherwise},  
    \end{cases}
$$
we have by, the definition of $\Jo$ and by the assumption $\omega\in(0,\pi)$,
 the corresponding solutions $\alpha_1,\alpha_{\Jo}$ and $\alpha_{\Jo+1}$ to belong to $[0,1]$ and that they also satisfy (\ref{meq}).
\end{proof}

\begin{figure}[ht]
\captionsetup[subfloat]{font=small,margin=0pt,labelfont={sc},labelsep=newline,justification=centering}
\begin{center}
\subfloat[\tiny $\omega=\pi/3$,\newline\noindent $|\mathbf f_1|=|\mathbf f_{\Jo}| =|\mathbf f_{\Jo+1}|=1$]{
 \begin{picture}(100,230)
\put(0,0){\includegraphics[scale=0.2]{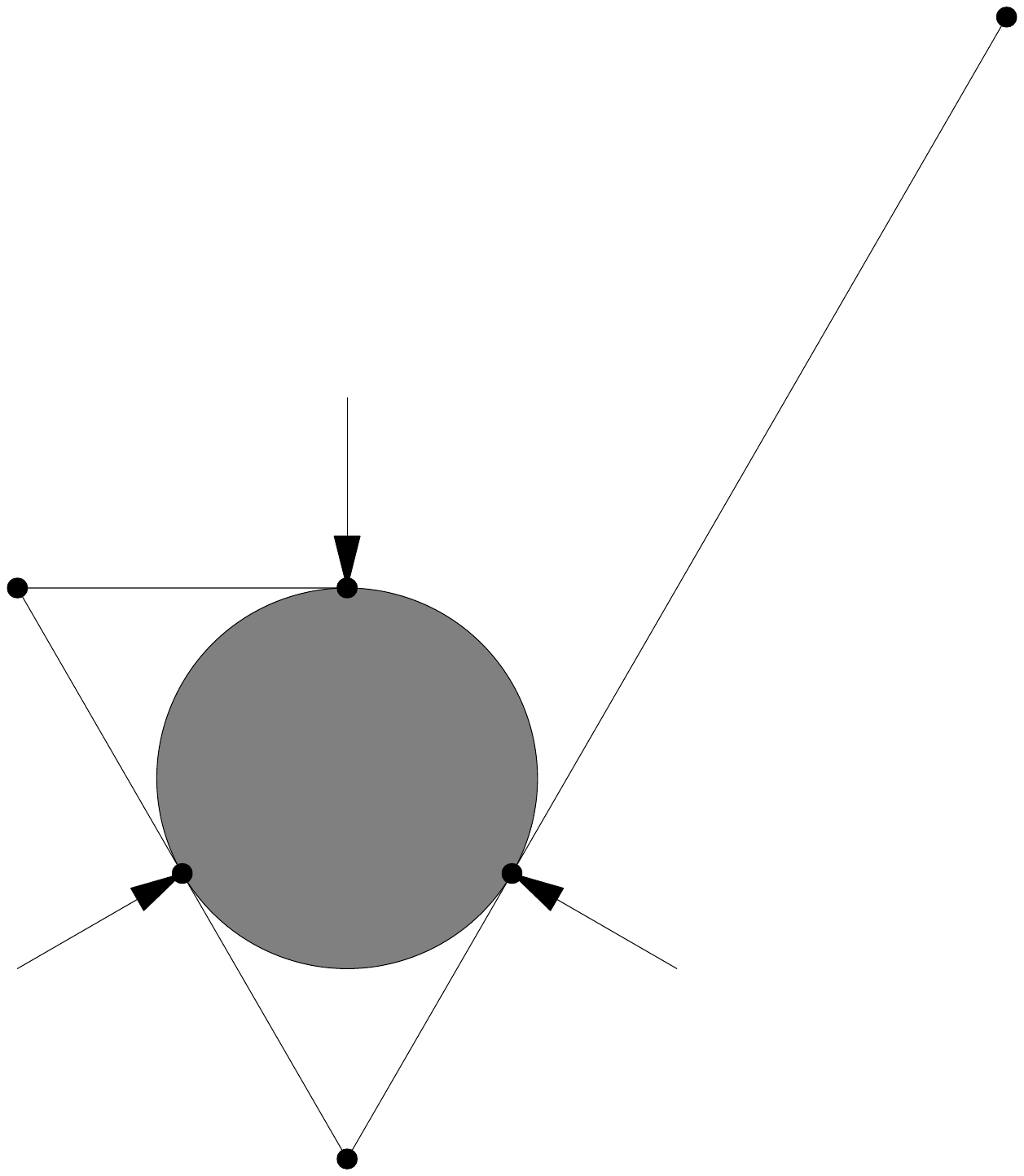}}
\put(10,95){$\mathbf f_{\Jo+1}$}
\put(0,45){$\mathbf f_{\Jo}$}
\put(70,45){$\mathbf f_{1}$}
\put(72,95){$C_1$}
\put(55,155){$x_0$}
\put(5,175){$x_{\Jo+1}$}
\put(-20,100){$C_{\Jo+1}$}
 \end{picture}
}\hskip0.5cm
\subfloat[\tiny $\qquad\omega=\pi/4$, \newline $|\mathbf f_1| =|\mathbf f_{\Jo+1}|=1$, $|\mathbf f_{\Jo}|=0$]{ \begin{picture}(100,230)
\put(0,0){\includegraphics[scale=0.2]{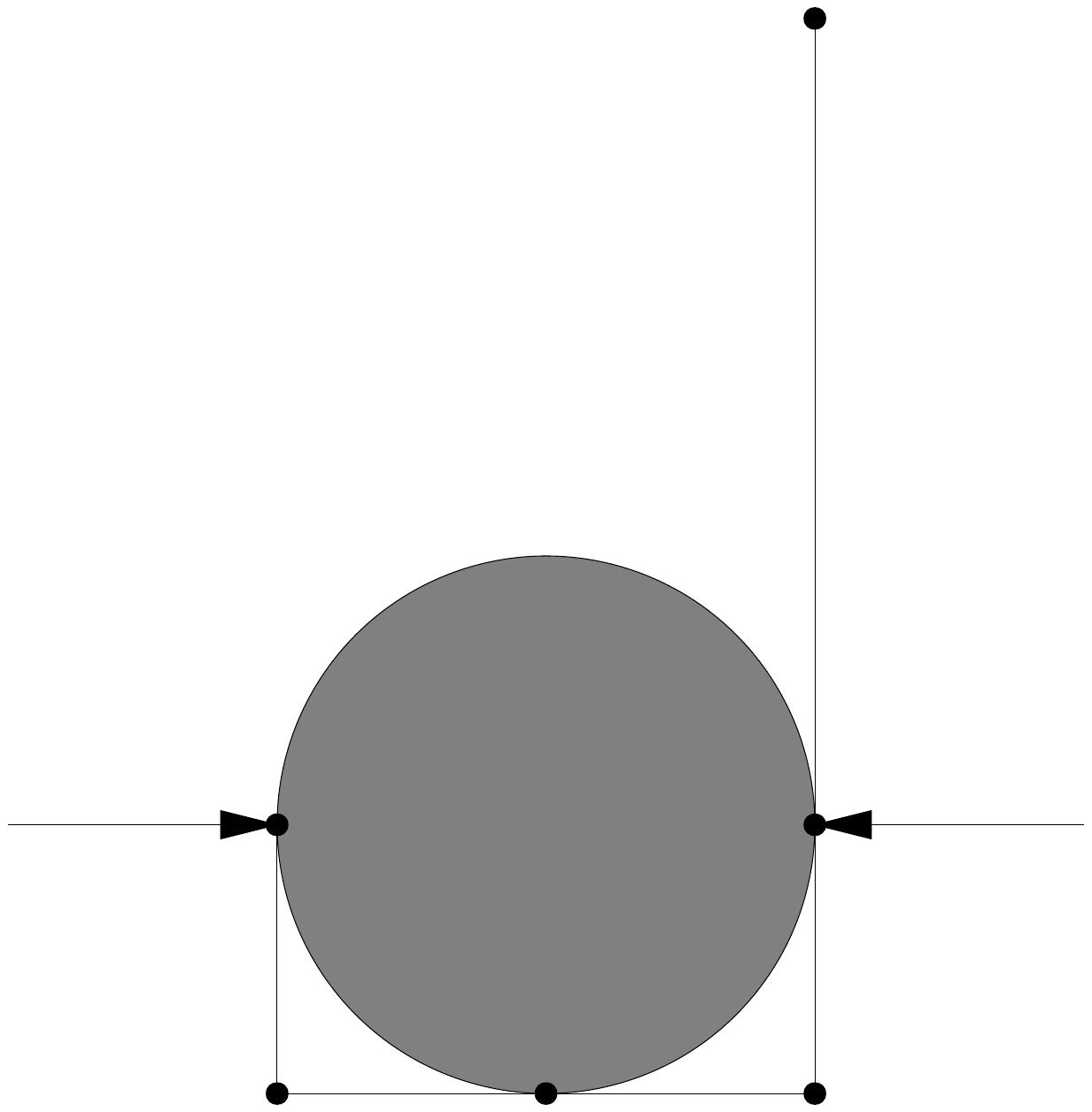}}
\put(95,55){$\mathbf f_1$}
\put(2,55){$\mathbf f_{\Jo+1}$}
 \end{picture}\hskip0.5cm}
 \subfloat[\tiny $\qquad\omega=\pi/5$, \newline $|\mathbf f_1|=1$, $|\mathbf f_{\Jo}| =|\mathbf f_{\Jo+1}|=1/G$, where $G$ is the Golden Ratio. ]
{\begin{picture}(100,230)
\put(0,0){\includegraphics[scale=0.2]{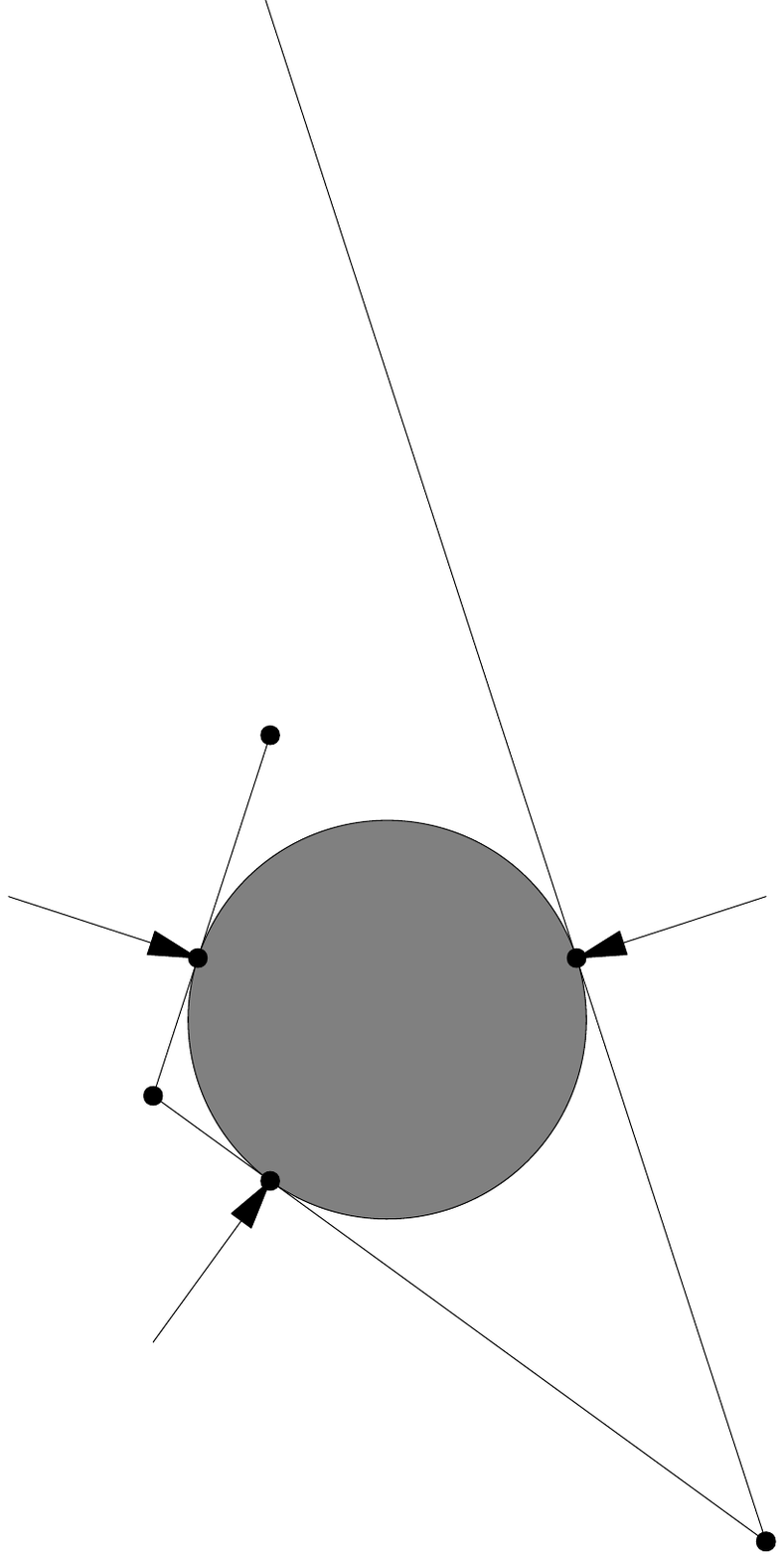}}
\put(38,38){$\mathbf f_{\Jo}$}
\put(90,70){$\mathbf f_1$}
\put(20,75){$\mathbf f_{\Jo+1}$}
 \end{picture}}
\end{center}
\caption{\label{figFor}}
 \end{figure}
 \subsection{Form closure and self-similarity}
In Section \ref{sself} we recalled that the complex reachable workspace $\Rhc$ is the attractor the IFS $\mathcal F_{\rho,\omega}$ defined in (\ref{mathf}). This implies that every reachable point can be obtained by an appropriate concatenation of the linear maps:
\begin{center}
\begin{equation}\label{ff}
\begin{tabular}{ll}
$ f_1: x\mapsto \frac{1}{\rho}x \qquad$ & $f_2: x\mapsto \frac{e^{-i(\pi-\omega)}}{\rho}x$\\
 $f_3: x\mapsto \frac{1}{\rho}(x+1) \qquad$ & $f_4: x\mapsto \frac{e^{-i(\pi-\omega)}}{\rho}(x+1)$.
\end{tabular}
\end{equation}
\end{center}
To better understand the relation between this IFS and the control vectors, we introduce the map $d:\{0,1\}\times\{0,1\}\to\{1,2,3,4\}$ such that
\begin{center}
\begin{tabular}{ll} 
 $d(0,0)=1$&$d(0,1)=2$\\
 $d(1,0)=3$&$d(1,1)=4$\\
\end{tabular}
\end{center}
and, fixing the motion control vectors $(u_j)_{j=1}^J$ and $(v_j)_{j=1}^J$, we define the index sequence
$$d_j:=d(u_j,v_j)$$
for every $j=1,\dots,J$. Then the reachable point
$$x_J:=\sum_{j=1}^J\frac{u_j}{\rho^j}e^{-i\sum_{n=1}^j v_n \omega}$$
also satisfies the relation
\begin{equation}
 x_J=f_{d_J}\circ f_{d_{J-1}}\circ \cdots\circ f_{d_1}(0).
\end{equation}
Since $f_1,\dots,f_4$ are invertible maps, we also have
\begin{equation}
 0=f_{d_1}^{-1}\circ\cdots\circ f_{d_{J-1}}^{-1}\circ f_{d_{J}}^{-1}(x_J).
\end{equation}
\begin{Notation}
For every $h=1,2,3,4$ we define the map from $\RR^2$ onto itself
\begin{equation}\label{mathbff}
 \bar  f_h: (x,y)\mapsto(\Re(f_h(x+iy),\Im(f_h(x+iy)).
\end{equation}
\end{Notation}
\begin{Proposition}\label{pgraspmap}
 Let $\mathcal O\subset \RR^2$ and assume that there exists a triplet of control vectors 
$(\mathbf u,\mathbf v,\alpha)$ in $\{0,1\}^K\times\{0,1\}^K\times[0,1]^K$ such that $\mathcal O$ and the resulting configuration satisfy (C) (respectively (SC)) and (E). 
Let $d_j=d(u_j,v_j)$ for every $j= 1,\dots,K$. Then the scaled, translated and rotated copy of $\mathcal O$
\begin{equation}
\bar  f_{d_1}^{-1}\circ\cdots\circ\bar  f_{d_{J-1}}^{-1}\circ \bar  f_{d_{J}}^{-1}(\mathcal O).
\end{equation}
and the configuration corresponding to $((u_{J+j}),(v_{J+j}),\alpha_{J+j}))$ satisfy (C) (resp. (SC)) and (E). 

Moreover for every $((\bar u_j)_{j=1}^N,(\bar v_j)_{j=1}^N),0_N))\in\{0,1\}^N\times\{0,1\}^N\times\{0\}^N$ the 
 control vectors $\mathbf u=(\bar u_1,\dots,\bar u_N,u_1,\dots,u_K)$, $\mathbf v=(\bar v_1,\dots,\bar v_N,v_1,\dots,v_K)$ and 
$\alpha=(0_N,\alpha_1,\dots,\alpha_N)$ yield a configuration satisfying (C) (resp. (SC)) and (E) for the object
\begin{equation}
 \bar  f_{d_{N+K}}\circ  \cdots\circ \bar  f_{d_1}(\mathcal O).
\end{equation}
\end{Proposition}
\begin{figure}[ht]
\begin{center}
\includegraphics[scale=0.6]{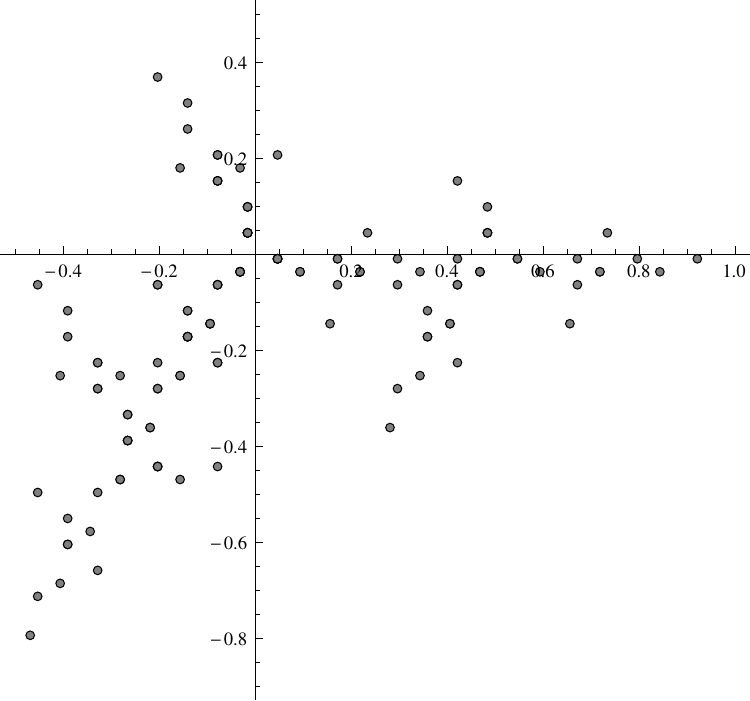} 
\end{center}
\caption{\label{graspmap} Every circle in this figure is of the form $\bar  f_{d_1}\circ\cdots \circ \bar  f_{d_4}(O)$ 
where $\mathcal O$ is the circle described in Theorem \ref{thgrasp}, with $\omega=\pi/3$, $\rho=2$ and $d_1,\dots,d_4\in\{1,2,3,4\}$. For every circle
in (A) there exists a configuration satisfying (SC) and (E). Moreover the angle between two consecutive phalanxes 
is either $\pi,\pi/6$ or $5\pi/6$, consequently no phalanx can intersect the scaled copy of $O_{\rho,\omega}$ without intersecting another phalanx. But this case is excluded
by Theorem \ref{thintersect}.} 
\end{figure}

\begin{Remark}
 The configuration described in the second part of the above result does not prevent in general the first $N$ phalanxes to touch or intersect the object. 
Nevertheless in the case $2\leq \rho< 3$ and $\omega=\pi/3$, namely when there not exist self-intersecting configurations (see Theorem \ref{thintersect}) 
and conditions of Theorem \ref{thgrasp} are satisfied, by iteratively applying the maps $\bar  f_1,\bar  f_2,\bar  f_3$ and $\bar  f_4$ on the circle $\mathcal O$ described in Theorem \ref{thgrasp}
one can construct a set of circles for which there exists a configuration with the form closure property, see Figure \ref{graspmap}. 
 \end{Remark}
\subsection{Form closure for three-dimensional objects: some examples}

In this section we show two configurations involving the whole hand.
 The general settings are the following: the hand has $5$ fingers ($H=5$), the angle between the phalanxes is constantly $\pi/2$, the
scaling ratio is $\rho=2$ for all the fingers. The distance between phalanxes is assumed to be constant.

 In the first example we are interested in a cylinder whose axis is normal to the planes of the fingers and 
whose section $O$ is a rescaling and a translation of a circle, $O_\omega$, in particular
\begin{equation}\label{eqO}
O=\bar  f_3(O_\omega)
\end{equation}
where $\bar  f_3$ is like in (\ref{mathbff}) -- see also (\ref{ff}).
 Remark that $O_\omega$ is well defined by virtue of Theorem \ref{thgrasp} and the related configuration is 
$u_j=v_j=1$ for $j=1,2,3$ and $\alpha_1=\alpha_3=1$ and $\alpha_2=0$. Since $d(1,0)=3$, by virtue of Proposition \ref{graspmap}
 then the configuration where the digit $1$ is perpended to extension control vector and 
the digit $0$ is perpended to the rotation control vector, is a suitable configuration for $O$ -- see (\ref{eqO}).
 It is easy to verify that this configuration also satisfies the contact constraints and if we apply it to all the fingers of the hand we obtain a configuration satisfying form closure for the
cylinder whose section is equal to $O$ -- see Figure \ref{figghand1}.
\begin{figure}
\begin{center}
\subfloat[$u_1=1,~v_1=0$]{\hskip-0.0cm\includegraphics[scale=0.1]{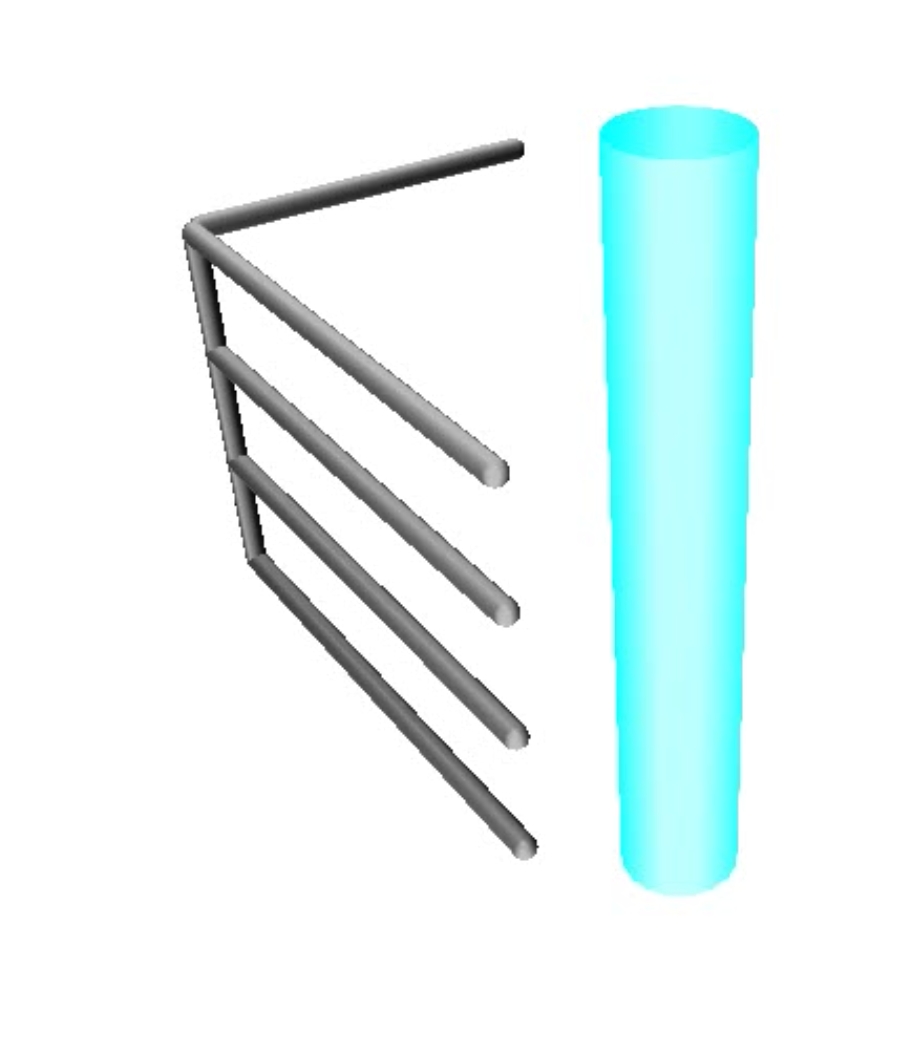}}
 \subfloat[$u_2=v_2=1$]{\hskip-0.5cm\includegraphics[scale=0.1]{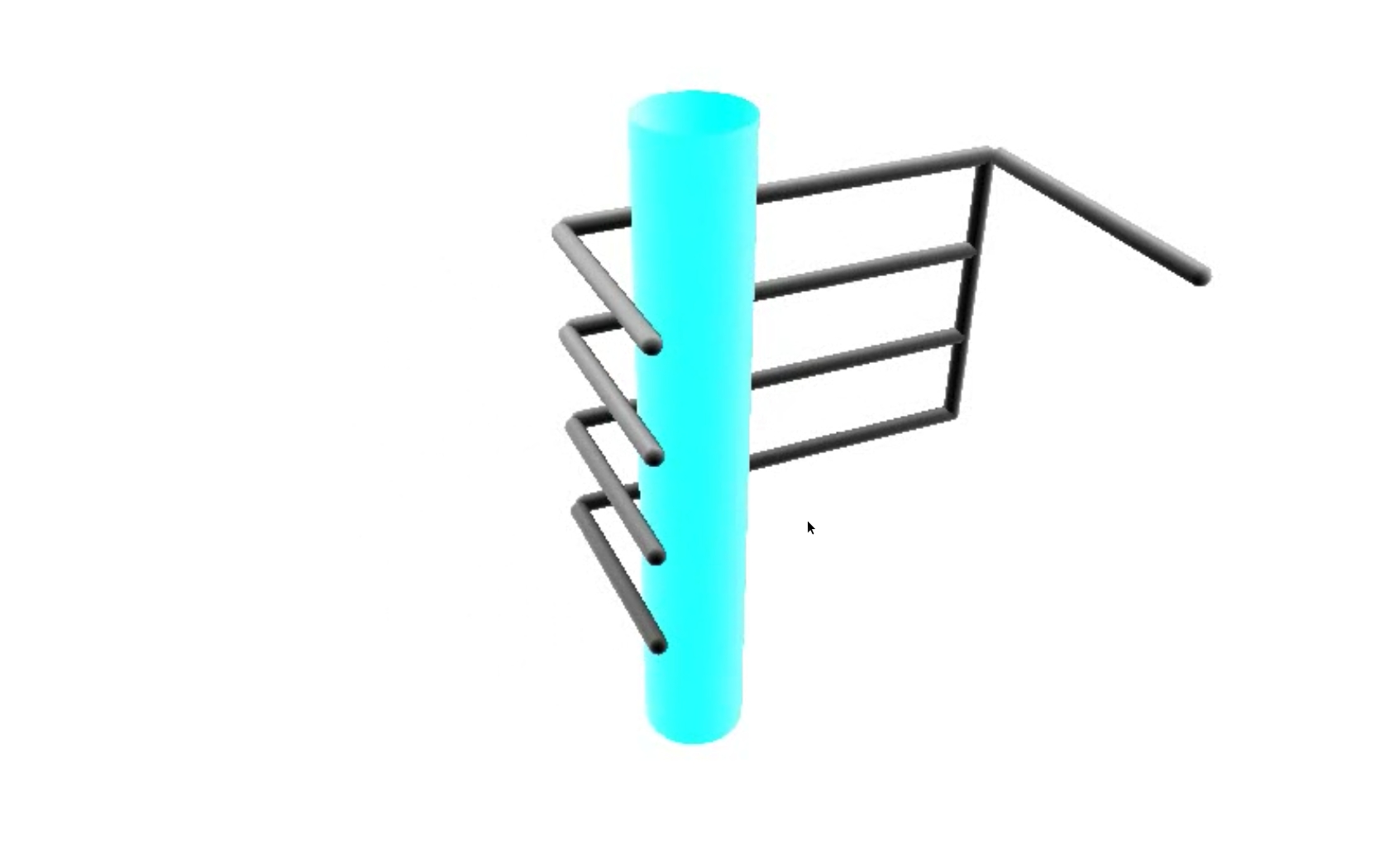}}
 \subfloat[$u_3=v_3=1$]{\includegraphics[scale=0.1]{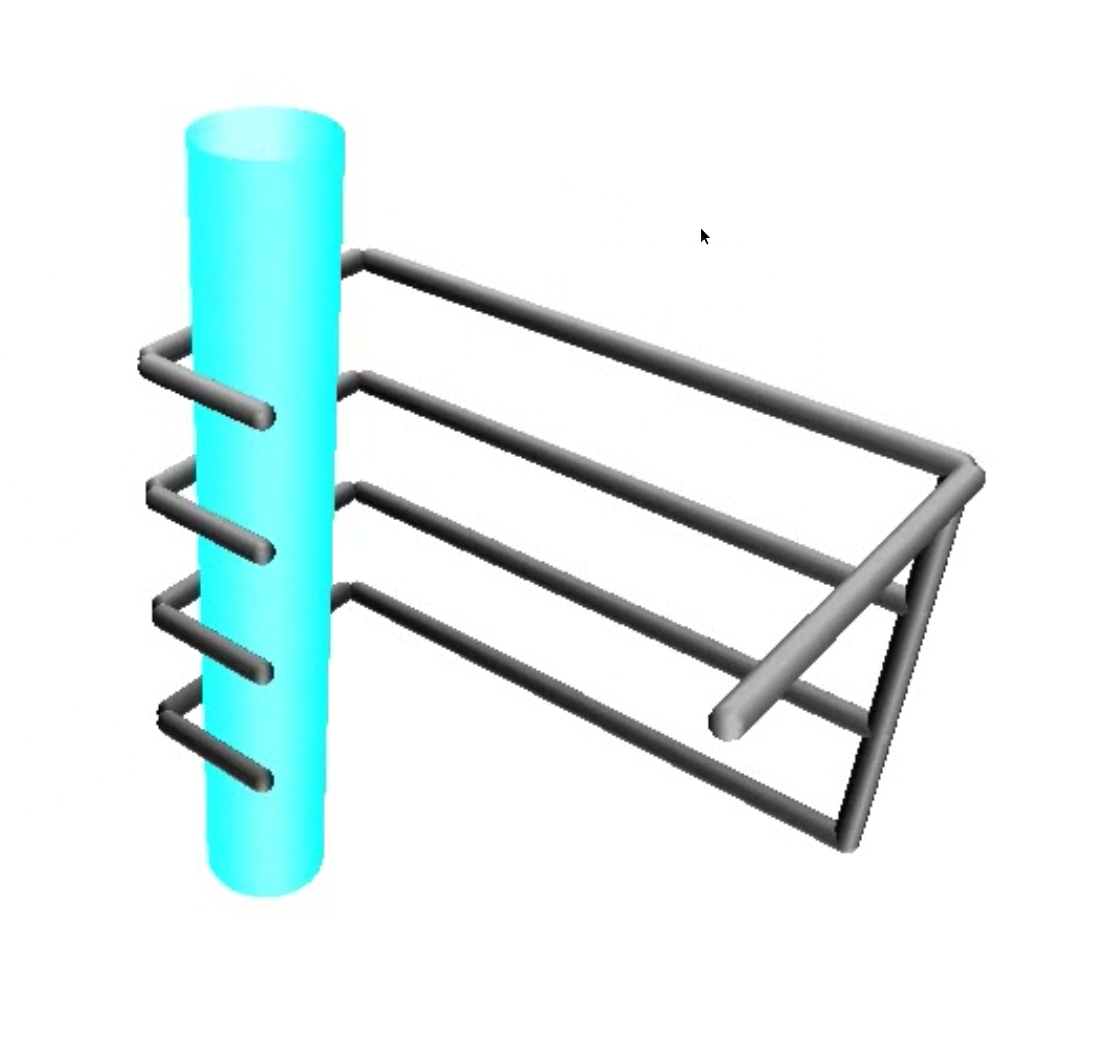}}\\
 \subfloat[$u_4=v_4=1$]{\includegraphics[scale=0.1]{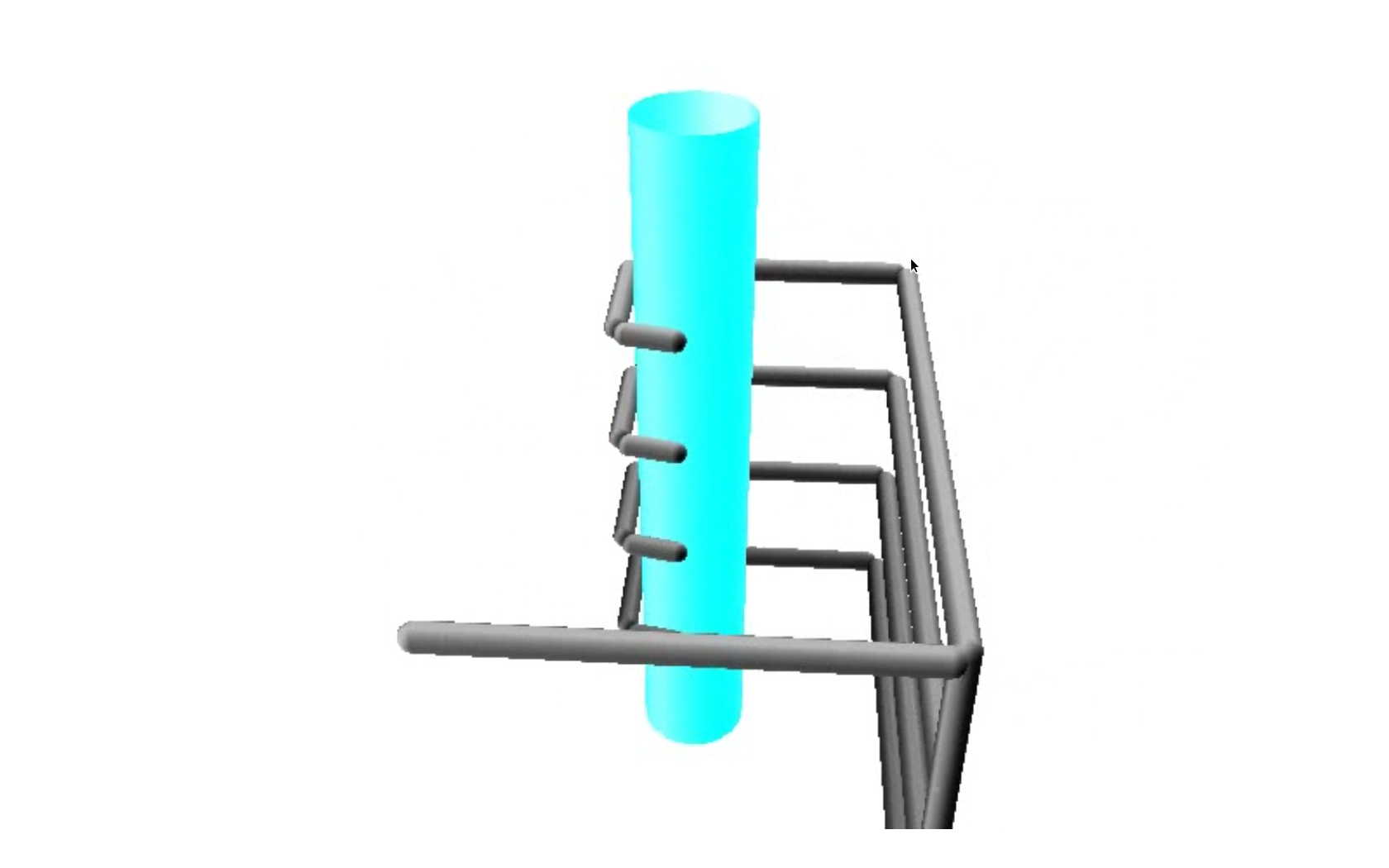}}
 \subfloat[$u_5=1$, $v_5=0$]{\includegraphics[scale=0.1]{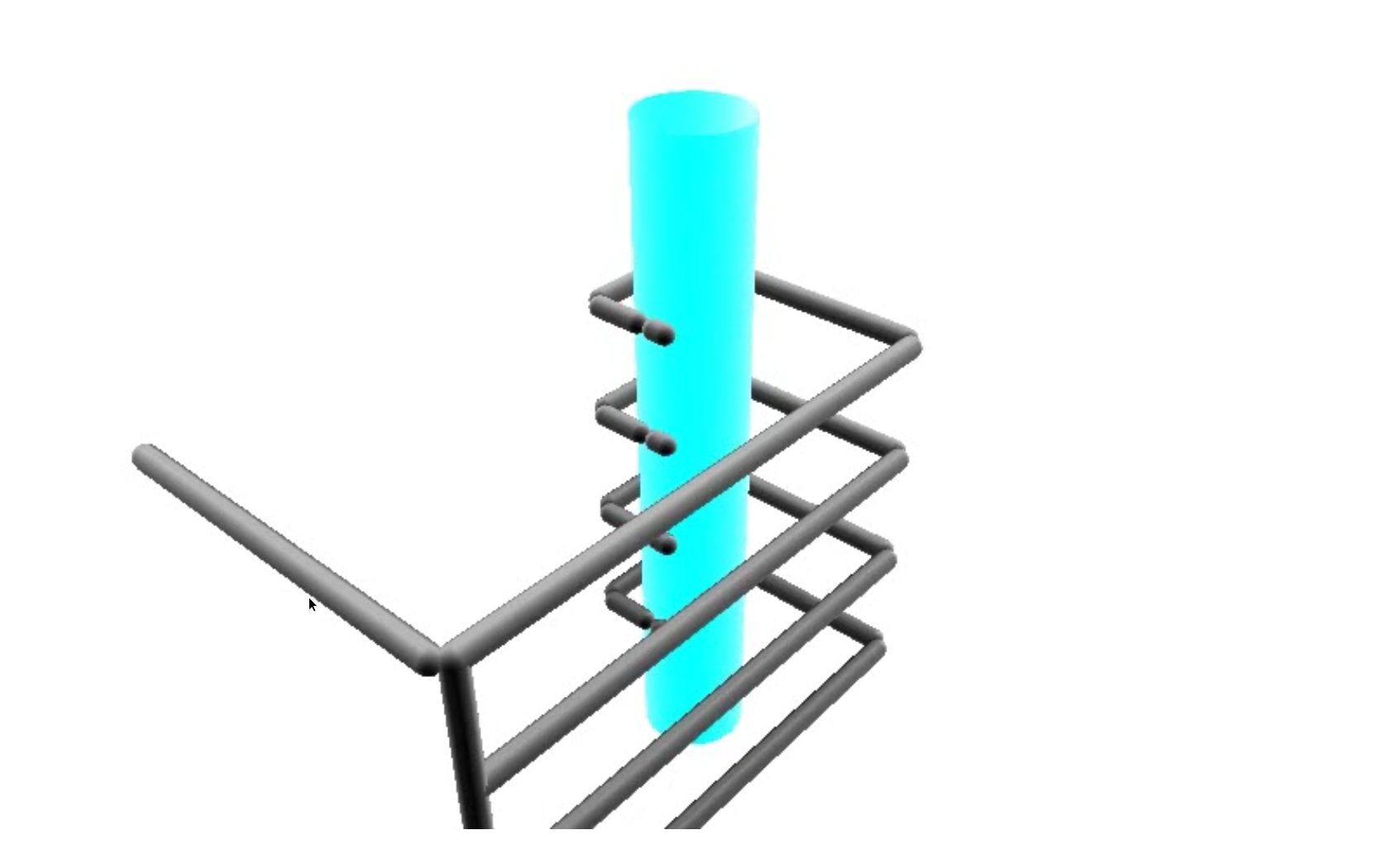}} 
\end{center}
 \caption{\label{figghand1} Various stages of the manipulation of a cylinder, whose section is $\bar  f_3(O_\omega)$. Due to numerical and graphical reasons, 
we extended one more phalanx with respect the configuration described in Theorem \ref{thgrasp}, 
so that for every finger the resulting motion control vector is $u_1=u_4=1$, $v_1=v_5=0$ and $u_j=v_j=1$ with $j=2,3,4$. }
\end{figure}

In Figure \ref{figghand2} we consider a cylinder whose section is a circle after a rotation of $2\pi/3$ along the $Oy$ direction.
 In this case we also use the opposable thumb, whose parameter $\omega$ is assumed to be different
from the others, in particular $\omega^{(1)}=\pi/4$. 
\begin{figure}
\begin{center}
 \includegraphics[scale=0.3]{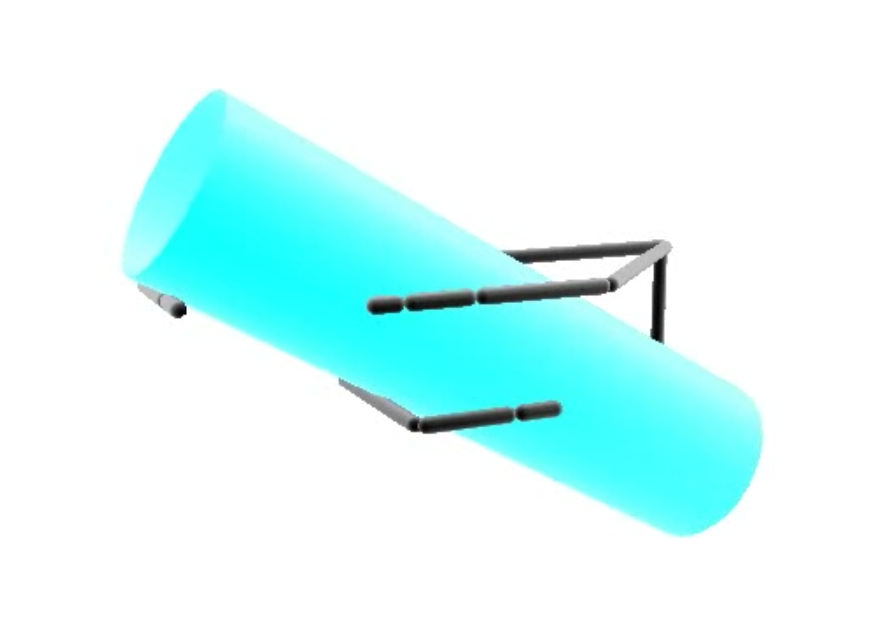} 
\end{center}
\caption{\label{figghand2} Thumb motion controls: $u_j=v_1=v_2=1$ for $j=1,2,3,4$; $v_j=0$ for $j=3,4$. \\Index finger motion controls: $u_j=v_3=1$ for $j=1,2,3,4$; $v_1,v_2,v_4=0$.\\
Last finger motion controls $u_j=v_2=v_3=1$ for $j=1,2,3,4$; $v_1,v_4=0$.}
\end{figure}
\begin{figure}
\begin{center}
\subfloat[]{\begin{picture}(120,160)
             \put(0,15){\includegraphics[scale=0.15]{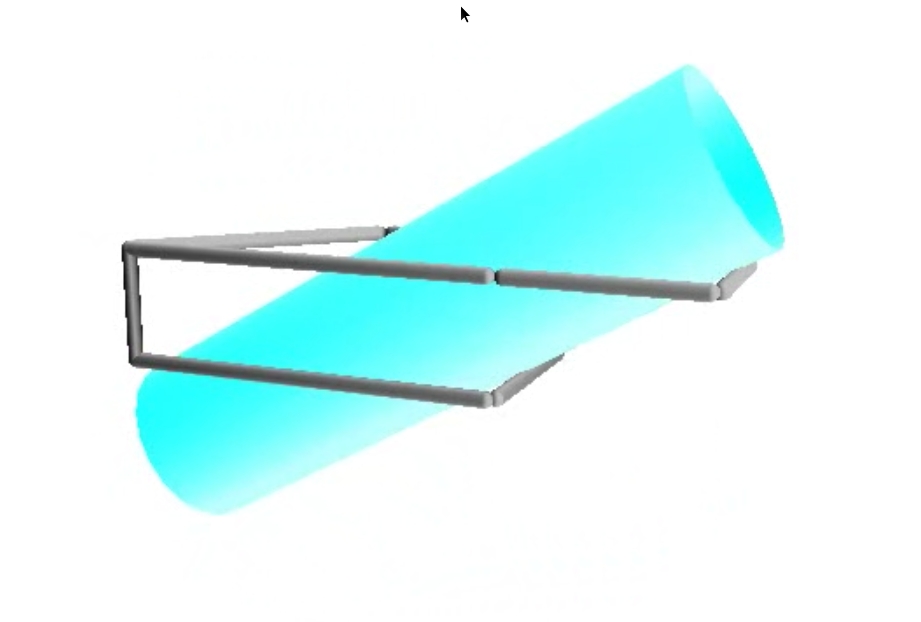}}
            \end{picture}}
\subfloat[]{\begin{picture}(120,160)
             \put(0,12){\includegraphics[scale=0.15]{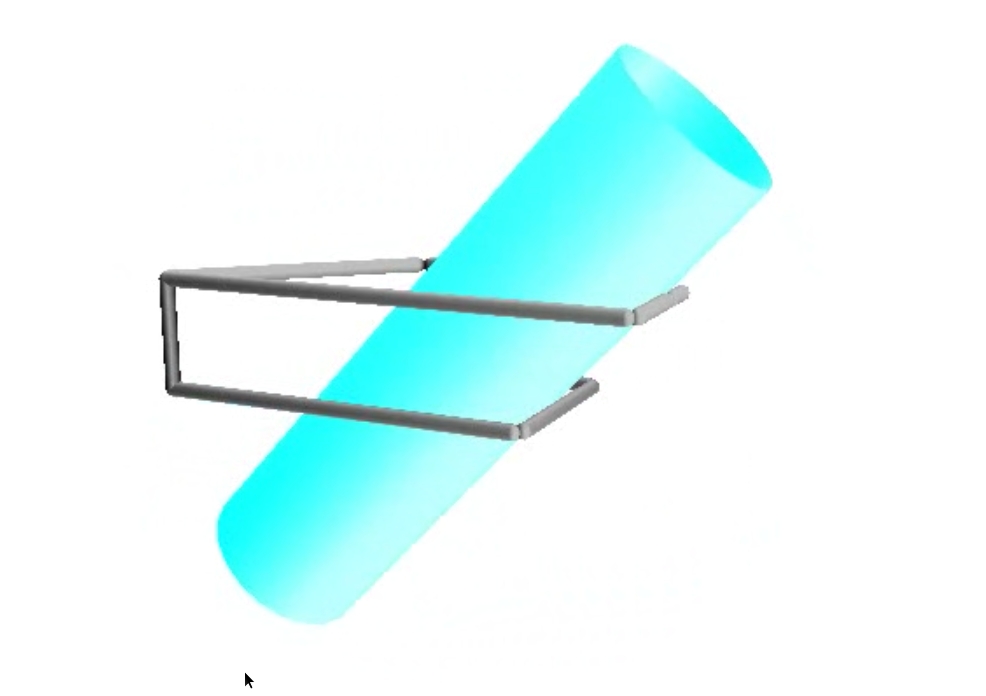}}
            \end{picture}}
\subfloat[]{\begin{picture}(120,160)
             \put(0,-25){\includegraphics[scale=0.15]{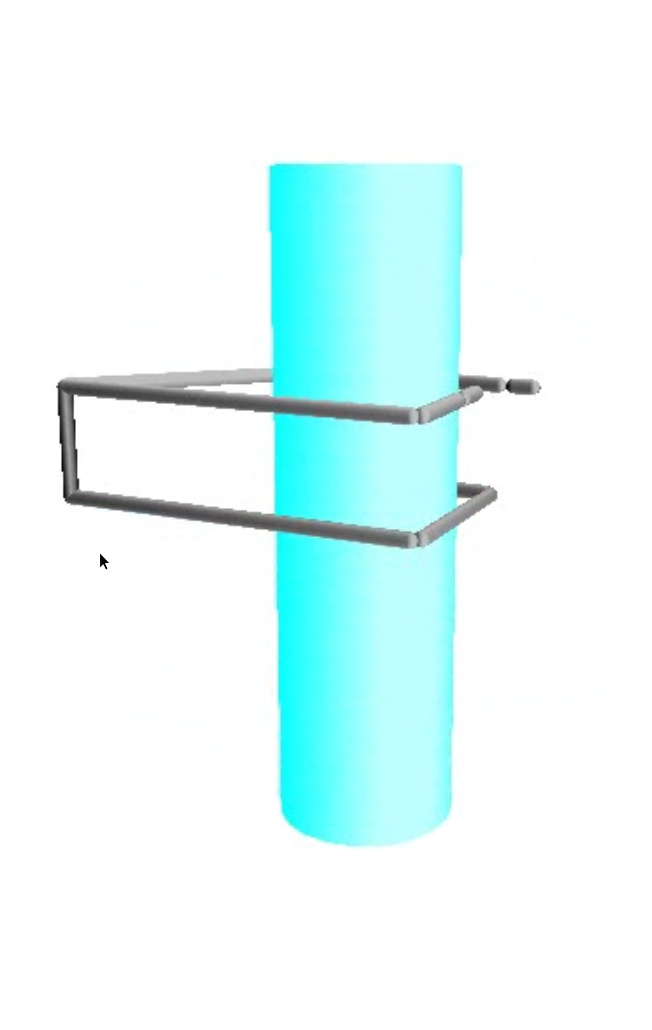}}
            \end{picture}}
\end{center}
\caption{Various stages of the manipulation of a cylinder}
\end{figure}

\newpage
\section{Conclusions}
We introduced a model for a robot hand whose fingers are planar manipulators with an arbitrary number of self-similar phalanxes, whose extension and rotations is controlled by binary actuators. 
We set the investigation of the workspace of every finger in the complex plane and we showed that its closure (with respect to the number of the phalanxes)
 is the attractor of an appropriate iterated function system. We also showed that a subset of the workspace, namely the points corresponding to full-rotation configurations, is indeed the 
set of expansions in an appropriate complex base with binary alphabet. This gives access of several results on non-standard numeration systems that we recalled at the end of Section 3 (for instance
the geometry of the workspace). We then used one of these results, the characterization of the convex hull of the workspace in a particular case, to establish some conditions on the parameters
of the fingers in order to avoid self-intersecting configurations.  An investigation of form closure properties is then approached: we considered some closure conditions and we described
a class of configurations satisfying them with respect to a suitable circle. We then used self-similarity and the iterated function system generating the reachable workspace 
to give an explicit form closure condition for a larger class of objects. Finally we showed some numerical simulations describing those configurations for three-dimensional objects.

Throughout this paper self-similarity and the connection with the theory of expansions in non-integer bases are our main tool of investigation.
 Although the problems studied in the present paper are well investigated in the literature, the novelty is the approach based on 
number theory and, more exactly, on the expansions in non-integer bases. This approach is suitable to threat the problem in a general context.  
 A variety of theoretical results can yet be applied in order to give more precise description of self-similar manipulators.
 For instance, a number of algorithms for the expansions in complex bases are available in the literature: with appropriate modifications they may be reread as inverse kinematics algorithms. 
Redundancy of the representations (namely redundancy of configurations reaching a given point), optimization on the digit sequences, geometrical investigations of the representable set 
(namely of the reachable workspace and/or some of its subsets) 
are active research domains on the field of non-standard numeration systems.

\end{document}